\documentclass[11pt,twoside,reqno,centertags]{amsart}
\usepackage{amsmath,amsthm,amsfonts,amssymb}
\advance\textheight by 6truemm
\newtheorem{Theorem}{Theorem}[section]

\newtheorem{theorem}[Theorem]{Theorem}
\newtheorem{corollary}{Corollary}[section]
\newtheorem{proposition}{Proposition}[section]
\newtheorem{lemma}{Lemma}[section]
\newtheorem{remark}{Remark}[section]

\newcommand{\R}{{\mathbb R}}
\newcommand{\eps}{\varepsilon}

\numberwithin{equation}{section}

\begin{document}

\title[A priori estimates]{A priori estimates of solutions to local \\
and nonlocal superlinear parabolic problems}
\author[Pavol Quittner]{Pavol Quittner}

\thanks{Author's address: Department of Applied Mathematics and Statistics, Comenius University
Mlynsk\'a dolina, 84248 Bratislava, Slovakia. Email: quittner@fmph.uniba.sk}

\date{}

\begin{abstract}
We consider a priori estimates of possibly sign-changing solutions
to superlinear parabolic problems and their applications
(blow-up rates, energy blow-up, continuity of blow-up time, existence of nontrivial steady states etc).
Our estimates are based mainly on energy, interpolation and bootstrap
arguments, but we also use the Pohozaev identity, for example.
We first discuss some known results on local problems 
and then consider problems with nonlocal nonlinearities
or nonlocal differential operators.
In particular, we deal with the fractional Laplacian and nonlinearities of Choquard type.
 
\vskip 0.2cm
{\bf AMS Classification:} 35K58; 35K61; 35B40; 35B45; 35B44; 35J66. 
\vskip 0.1cm

{\bf Keywords:} Nonlocal parabolic problem, a priori estimates, blow-up,
Choquard nonlinearity, fractional Laplacian, Schr\"odinger-Poisson-Slater system
\end{abstract}

\maketitle

\section{Introduction and main results}
\label{intro}

In this paper we consider parabolic problems of the type
\begin{equation} \label{AF}
\begin{aligned}
u_t+{\mathcal A}u &={\mathcal F}(u) &\quad&\hbox{in }\ \Omega\times(0,T), \\
u &=0 &\quad&\hbox{in }(\R^n\setminus\Omega)\times(0,T), \\
u(\cdot,0) &=u_0,
\end{aligned}
\end{equation}
where 
$u:\R^n\times[0,T)\to\R^N$, $\Omega\subset\R^n$,
$\mathcal A$ is a possibly nonlocal differential operator 
and $\mathcal F$ is a possibly nonlocal superlinear potential operator.
Assume that problem \eqref{AF} is locally well-posed in a Banach space
$X$ and  $u=u(x,t)=u(x,t;u_0)$ is a solution of \eqref{AF} with
$u_0\in X$ and the maximal existence time $T=T_{max}(u_0)\leq\infty$.
Given $\delta>0$ and $c_0>0$,
we are interested in the a priori estimate 
\begin{equation} \label{AE}
  \|u(\cdot,t;u_0)\|_X \leq C(\delta,c_0) \quad
  \hbox{whenever }\ t\in[0,T_{max}(u_0)-\delta)\ \hbox{ and }\ \|u_0\|_X\leq c_0,
\end{equation}
where $T_{max}(u_0)-\delta:=\infty$ if $T_{max}(u_0)=\infty$.
If one considers only global solutions (i.e. $T_{max}=\infty$), then \eqref{AE}
can be written in the simpler form
\begin{equation} \label{AE2}
\|u(\cdot,t;u_0)\|_X \leq C(c_0) \quad
  \hbox{whenever }\ t\in[0,\infty)\ \hbox{ and }\ \|u_0\|_X\leq c_0.
\end{equation}
Notice that if a global solution $u$ with initial data $u_0$ is bounded in $X$,
i.e.~$C(u_0):=\sup_{t\ge0}\|u(\cdot,t;u_0)\|_X<\infty$, then
$u$ satisfies the estimate
\begin{equation} \label{AEG}
 \|u(\cdot,t;u_0)\|_X \leq C(u_0), \quad t\in[0,\infty), 
\end{equation}
which is weaker than \eqref{AE2} 
and does not imply the consequences of estimates \eqref{AE} or \eqref{AE2}
that we are interested in, see Remark~\ref{rempS}.

We will first comment on known results for local problems
(Subsections~\ref{subsec1} and~\ref{subsec2})
and then consider problems with nonlocal nonlinearities
or nonlocal differential operators (Subsection~\ref{subsec3}).
Assuming that $\Omega\subset\R^n$ is bounded and smooth, our results 
guarantee estimate \eqref{AE} and its consequences
for problems with
$$ {\mathcal A}u=-\Delta u+\mu u\quad\hbox{and}\quad {\mathcal F}(u)=
\mu|u|^{q-1}u+
\lambda\Bigl(\int_\Omega\frac{|u(y)|^p}{|x-y|^{n-2}}\,dy\Bigr)|u|^{p-2}u,
$$
where either $\mu=1$, $\lambda<0$, $n=3$, $p=2$ and $q\in(2,5)$, 
or $\mu=0$, $\lambda=1$, $3\le n\le5$ and $p\in[2,p^*)$
(see  Theorem~\ref{thm23}, Corollary~\ref{cor23} and Theorem~\ref{thmChoquard}
for more precise statements),
and Theorem~\ref{thmFracLap} guarantees estimate \eqref{AE} provided
$$ {\mathcal A}u=(-\Delta)^\alpha u\quad\hbox{and}\quad {\mathcal F}(u)=|u|^{p-1}u,$$
where $\alpha\in(0,1)$ and $p>1$ is subcritical in the Sobolev sense.

\subsection{Early results and applications to blow-up rates and decay} \label{subsec1}

Let us first consider global solutions of the model scalar problem 
\begin{equation} \label{MP}
\begin{aligned}
u_t-\Delta u &=|u|^{p-1}u &\quad&\hbox{in }\ \Omega\times(0,\infty), \\
u &=0 &\quad&\hbox{on }\partial\Omega\times(0,\infty), \\
u(\cdot,0) &=u_0,
\end{aligned}
\end{equation}
where $\Omega\subset\R^n$ is bounded and $p>1$.
In 1984, Cazenave and Lions \cite{CL84} used energy and interpolation arguments
to show estimate \eqref{AE2} provided $p<p_{CL}$, where
$p_{CL}=\infty$ if $n=1$ and $p_{CL}=(3n+8)/(3n-4)$ otherwise,
and they also proved the weaker estimate \eqref{AEG} provided $p<p_S$,
where $p_S$ is the critical Sobolev exponent
$$ p_S=\begin{cases} \infty &\hbox{ if }n\leq2, \\
                    \frac{n+2}{n-2} &\hbox{ if }n>2 \end{cases}$$
(notice that $p_{CL}<p_S$ if $n>1$).
In fact, their results apply to 
more general local nonlinearities ${\mathcal F}(u)=f(u)$, where
$f\in C^1(\R,\R)$
satisfies the growth assumption 
\begin{equation}  \label{assCL1}
|f(u)|\le C(|u|+|u|^p)
\end{equation}
and the superlinearity condition 
\begin{equation} \label{assCL2}
f(u)u\ge(2+\eps)F(u), \ \hbox{ where }\ \eps>0 \hbox{ and }\ 
F(u)=\int_0^uf(s)\,ds.
\end{equation}
Estimate \eqref{AE2} in \cite{CL84} significantly improved a slightly earlier result
by Ni, Sacks and Tavantzis
\cite{NST84} for nonnegative solutions of \eqref{MP} with $\Omega$ convex and $p<p_F:=(n+2)/n$. 
The results in \cite{NST84} also imply that estimate \eqref{AE2} fails if
$\Omega$ is convex and $p\geq p_S$.

In 1986, Giga \cite{G86} used rescaling arguments and the elliptic
Liouville theorem in \cite{GS81} in order to prove \eqref{AE2} 
for nonnegative global solutions of \eqref{MP} under
the optimal growth condition $p<p_S$.
The validity of \eqref{AE2} for sign-changing global solutions of \eqref{MP} 
in this full subcritical range remained open.

Assume that a solution $u$ of the model equation
\begin{equation} \label{eqGK}
 u_t-\Delta u=|u|^{p-1}u \quad \hbox{ in }\R^n\times(0,T) 
\end{equation}
with initial data $u_0\in L_\infty(\R^n)$
blows up in finite time $T$. 
In 1987, Giga and Kohn \cite{GK87} used the backward similarity variables 
$(y,s)=((x-a)/\sqrt{T-t},-\log(T-t))$
with $a\in\R^n$
in order to transform the solution $u$ into a global solution 
$w(y,s)=(T-t)^{1/(p-1)}u(x,t)$ of  
the rescaled equation 
\begin{equation} \label{eq-w1}
 w_s =\Delta w-\frac12 y\cdot\nabla w-\frac1{p-1} w+|w|^{p-1}w 
  \quad\hbox{in } \R^n\times(-\log T,\infty),
\end{equation}
and then, using the arguments in
\cite{CL84} and \cite{G86}, they proved uniform bounds 
for solutions $w$ provided
\begin{equation} \label{rangeGK}
\hbox{either }\ p<p_{CL} \ \hbox{ or }\ (u\geq0 \hbox{ and }p<p_S).
\end{equation}
Assuming \eqref{rangeGK},
the bounds on $w$ imply the optimal blow-up rate estimate
\begin{equation} \label{rateGK}
\|u(\cdot,t)\|_\infty\leq C(T-t)^{-1/(p-1)}
\end{equation}
with $C=C(u)$.
 
Assume now that $u$ is a solution of \eqref{eqGK} with $T=\infty$
and initial data $u_0$ belonging to the weighted Sobolev space $X:=\{u: u,\nabla u\in L_2(\R^n,e^{|x|^2/4}dx)\}$.  
In 1987, Kavian \cite{K87} used the forward similarity variables 
$(y,s)=(x/\sqrt{t+1},\log(t+1))$
to transform the solution $u$ into the solution 
$w(y,s)=(t+1)^{1/(p-1)}u(x,t)$ of the rescaled equation  
$$ w_s =\Delta w+\frac12 y\cdot\nabla w+\frac1{p-1}w+|w|^{p-1}w 
  \quad\hbox{in } \R^n\times(0,\infty), $$
and then used the arguments in \cite{CL84} 
to show the boundedness of $w$ provided $p_F<p<p_S$.
Consequently, he obtained the decay estimate
\begin{equation} \label{decayFwd}
\|u(\cdot,t)\|_\infty\leq Ct^{-1/(p-1)} \quad\hbox{for }\ t\geq1,
\end{equation}
where $C=C(u_0)$.

Due to the importance of estimate \eqref{AE2},
assumption $p<p_{CL}$ appeared --- in addition to \cite{GK87} ---
in several other papers, see
\cite{E86}, \cite{FZ00} or \cite{MZ00}, for example.
The fact that the condition $p<p_{CL}$ is just technical (at least
in the case of \eqref{MP} with $\Omega$ bounded) was shown by the author \cite{Q99} in 1999:
He used a bootstrap argument to show that estimate \eqref{AE2} for global
solutions of \eqref{MP} is true for all $p\in(1,p_S)$. 
In 2004, by using this bootstrap argument and solving nontrivial technical issues, 
Giga, Matsui and Sasayama \cite{GMS04}  proved the corresponding estimate
also for the rescaled equation \eqref{eq-w1}, 
hence established the blow-up rate estimate \eqref{rateGK}
for sign-changing solutions of \eqref{eqGK} in the full subcritical range $p\in(1,p_S)$. 
By solving still more technical problems, 
the analogue of their result was later proved also for nonlinearities
of the form $f(u)=|u|^{p-1}u+h(u)$ or $f(u)=|u|^{p-1}u\log^a(2+u^2)$
(see \cite{N15} and \cite{HZ22}).
In fact, although the bootstrap argument in \cite{Q99} can be applied
to rather strong perturbations of the power nonlinearity $|u|^{p-1}u$
(see below), even its weak perturbation produces a nonautonomous
nonlinearity in the rescaled equation \eqref{eq-w1} and this causes serious problems
in finding a suitable Lyapunov type functional. 
Arguments in \cite{GMS04} have also been used for parabolic systems
with nonlinearities of the form
$f(u)=\nabla(\sum\beta_{ij}|u_i|^\frac{p+1}2|u_j|^\frac{p+1}2)$ with
$\beta_{ij}\geq0$
\cite{Z21}, and for problems in convex domains \cite{GMS04a,Z22}.  
The bootstrap argument in \cite{Q99} 
also allows one to prove the decay estimate \eqref{decayFwd}
with $C=C(\|u_0\|_X)$, see \cite[Remark 18.5]{QS19}.

\begin{remark} \label{remFwd} \rm {\bf Universal estimates.}
Consider classical solutions $u$ of \eqref{AF} which do not need to satisfy
the initial condition $u(\cdot,0)=u_0$. In some cases one can prove the estimate 
\begin{equation} \label{UEdelta}
  \|u(\cdot,t)\|_\infty \leq C(\delta) \quad
  \hbox{whenever }\ t\in(\delta,T-\delta),
\end{equation}
where $\delta>0$, $T-\delta=\infty$ if $T=\infty$, and
the {\it universal}\/ constant $C=C(\delta)$ can also depend on ${\mathcal A},{\mathcal F}$
and $\Omega$, but is independent of $u$ and $T$.
Such estimate for positive global solutions of \eqref{MP} with $\Omega$
bounded were obtained in \cite{FSW01,Q01,QSW04}, for example.
Later it turned out that by using Liouville-type theorems combined
with scaling and doubling arguments, one can prove such universal estimates
with an explicit (and optimal) dependence of $C(\delta)$ on $\delta$.
Assume, for example, that $p<p_S$ and $u$ is a solution of \eqref{eqGK} satisfying
$$
\hbox{either $u\geq0$ \ or \ $u(\cdot,t)$ is radially symmetric and $z(u(\cdot,t))\le k$,} 
$$
where $k$ is an integer and $z(v)$ denotes the number of sign changes of $v$.
Set also $k=0$ if $u\ge0$.
Then the scaling and doubling arguments in \cite{PQS07}
and the Liouville theorems in \cite{BPQ11,Q21} (cf. also \cite{Q22,QS24}) imply
the existence of a constant $C=C(n,p,k)$ such that
\begin{equation} \label{UE}
 |u(x,t)|\leq C(t^{-1/(p-1)}+(T-t)^{-1/(p-1)})\quad\hbox{for }\ (x,t)\in\R^n\times(0,T), 
\end{equation}
where $(T-t)^{-1/(p-1)}:=0$ if $T=\infty$,
hence the constant $C(\delta)$ in \eqref{UEdelta} can be written in the form
$C(\delta)=2C(n,p,k)\delta^{-1/(p-1)}$.
In particular, estimate \eqref{UE} guarantees both the decay and blow-up rate estimates 
mentioned above.
In addition, unlike the approach in \cite{GK87,GMS04}, these arguments can
also be used to prove blow-up rate estimate \eqref{rateGK} for nonnegative solutions
of the Dirichlet problem in nonconvex domains $\Omega$,
and \eqref{UE} also guarantees initial blow-up rate estimate as
$t\to0$ (see \cite{BV98} for early results of this type 
and cf.~also \cite[Theorems~26.13--26.14]{QS19}).

Notice that in the case of sign-changing solutions, decay estimates in \eqref{UE} or \eqref{decayFwd}
require either radial symmetry and finite number of sign changes or exponential spatial decay
of the solution, respectively. However, some assumptions of this type are also necessary:
If $n=1$, for example, then equation \eqref{eqGK} possesses nontrivial radial bounded 
steady states for any $p>1$.
In the case of the blow-up rate estimate \eqref{rateGK}, if we allow $C=C(u_0)$, then 
it is true for any $u_0\in L_\infty(\R^n)$ due to \cite{GMS04},
but if we wish to have a universal constant or a constant $C=C(\|u_0\|)$, 
then such estimate also requires additional assumptions on $u$.
In fact, if $n=1$ and $p>1$, then there exists a positive solution
$v_1$ of the equation $v_{xx}+v^p=0$ in $(0,1)$ satisfying the boundary
conditions $v(0)=v(1)=0$, and the solution $v_\lambda$ of the corresponding
parabolic problem with initial data $\lambda v_1$ is time-increasing and blows up in finite time
$T_\lambda$ if $\lambda>1$; in addition $T_\lambda\to\infty$ as $\lambda\to1+$.
Let $u_1(x):=v_1(x)$ for $x\in[0,1]$, $u_1(x):=-v_1(-x)$ for $x\in[-1,0)$,
and $u_1(x+2k)=u_1(x)$ for any integer $k$.
If $\lambda>1$, then the solution $u$ of \eqref{eqGK} with initial data $\lambda u_1$ blows up
at $T=T_\lambda$ and $\|u(\cdot,t;\lambda u_1)\|_\infty\geq\|v_1\|_\infty$ for
any $t\in[0,T_\lambda)$,
while \eqref{rateGK} with a universal constant $C$ would imply
$\|u(\cdot,T_\lambda/2;\lambda u_1)\|_\infty\le C(T_\lambda/2)^{-1/(p-1)}\to0$
as $\lambda\to1+$. Notice also that $\|\lambda u_1\|_\infty\leq C$ if $\lambda\in(1,2)$, for example,
hence the blow-up rate estimate \eqref{rateGK} cannot be true with a constant $C=C(\|u_0\|_\infty)$.
Similarly, given $T\in(0,\infty)$ fixed and $\lambda>1$, set
$\alpha=\alpha(\lambda):=\sqrt{T_\lambda/T}$. Then the solution $u$ with
initial data $u_0(x):=\alpha^{2/(p-1)}\lambda u_1(\alpha x)$ blows up at time $T$
and $\|u(\cdot,t)\|_\infty\geq\|u_0\|_\infty\to\infty$ as $\lambda\to1+$,
hence the universal estimate \eqref{UE} fails even if $T$ is fixed. 
\qed
\end{remark}

\subsection{Generalizations and further applications: Local problems } \label{subsec2}

The bootstrap argument in \cite{Q99} cannot be used 
under the general assumptions of Cazenave and Lions \eqref{assCL1}--\eqref{assCL2}
on the nonlinearity $f$, but it still allows quite strong perturbations and modifications
of the power nonlinearity ${\mathcal F}=|u|^{p-1}u$ (and the operator ${\mathcal A}=-\Delta$),
and the assumptions $\Omega$ bounded and $T_{max}=\infty$ are also not necessary.
In fact, this bootstrap argument was exploited in the author's paper \cite{Q03} 
in order to prove estimate \eqref{AE}
 for general second-order operators ${\mathcal A}$, $T\leq\infty$, $N=1$,
 $\Omega$ bounded or $\Omega=\R^n$, and  ${\mathcal F}(u)=f(\cdot,u)$  with
\begin{equation} \label{growthQ03}
c_1|u|^{p_1}-a_1(x)\leq (f(x,u)+\lambda u)\hbox{sign}(u)\leq c_2|u|^{p_2}+a_2(x),
\end{equation}
where $1<p_1<p_2<p_S$, $p_2-p_1<\kappa$, $\kappa>0$ is an explicit constant depending on $p_2$,
and $\lambda>0$ if $\Omega=\R^n$. 
Estimate \eqref{AE} was then used in \cite{Q03} to prove blow-up of energy and continuity
of the maximal existence time $T_{max}=T_{max}(u_0)$.
In a subsequent paper \cite{Q04}, estimate \eqref{AE} from \cite{Q03} was also used in the proof
of existence of nontrivial steady states or periodic solutions to related problems.
Similar applications in the study of steady states and connecting orbits
can be found in paper \cite{ABKQ08} which deals with equations
of the form $u_t-\Delta u=a(x)|u|^{p-1}u+h(x,u)$, where $a$ may change sign
and $h$ has at most linear growth in $u$.
However, in this case a modification of the bootstrap arguments in \cite{Q99}
was required and estimate \eqref{AE} was proved only if $n>2$ and $p<\hat p$,
where $\hat p\in(p_{CL},p_S)$ is an explicit constant. 
Finally, bootstrap arguments from \cite{Q99} were also used in \cite{AQ05}
in the proof of existence of optimal controls for problems
with final observation governed by superlinear parabolic equations,
and the basic idea from \cite{Q99} was also used in \cite{QS03,CQ04}
to prove \eqref{AE}
in the case of problems with nonlinear boundary conditions of the form
$\frac{\partial u}{\partial\nu}=|u|^{p-1}u$ on $\partial\Omega\times(0,T)$.
 
The results and arguments in \cite{CL84,Q99,Q03} have been used 
in the study of local problems by many other authors.
The topics considered by these authors include 
convergence to equilibria \cite{BJP02},
periodic solutions \cite{H02,HHJY15},
properties of the maximal existence time \cite{CR04,AFPR10}, 
invariant sets \cite{AB05,ABK07},
solutions with high energy initial data \cite{GW05},
saddle-point dynamics \cite{LS05},
multiplicity of equilibria \cite{D07,WW08,DMIP13,LW21},  
global dynamics of blow-up profiles \cite{FM07},
structure of global solutions \cite{CDW09,CDW09a},
Morse theory for indefinite elliptic problems \cite{CJ09},
blow-up with bounded energy in the critical and supercritical cases \cite{D19} etc. 

\subsection{Nonlocal problems} \label{subsec3}

A priori estimates of the form \eqref{AE}--\eqref{AEG} or \eqref{UEdelta}
have also been studied in the case of nonlocal problems.
For example, estimate \eqref{AEG} guaranteeing the boundedness of global
solutions was studied in \cite{F97,FL97,BL97,K04}
and the universal estimate \eqref{UEdelta} in \cite{R03,BCL17}.
Since estimate \eqref{AEG} does not imply the properties
that we are interested in (see Remark~\ref{rempS}) 
and known results on estimate \eqref{UEdelta}
are restricted to nonnegative solutions,
we will discuss only estimates \eqref{AE}--\eqref{AE2}.

Assuming
\begin{equation} \label{assNP1}
{\mathcal A}=-\Delta, \quad  N=1, \quad \Omega \ \hbox{ bounded},
\end{equation}
and using the arguments in \cite{CL84,Q99},
a priori estimates \eqref{AE} have been proved in \cite{Q03} 
for the nonlocal nonlinearities
$${\mathcal F}_1(u)=\frac{|u|^{p-1}u}{(1+\int_\Omega|u|^{p+1})^\alpha}, \quad 0<\alpha<\frac{p-1}{p+1},
 \quad p\in(1,p_S),$$ 
and
$${\mathcal F}_2(u)=|u|^{p-1}u-\frac1{|\Omega|}\int_\Omega|u|^{p-1}u, \quad p\in(1,p_S),$$
where the Dirichlet boundary condition $u=0$ on $\partial\Omega\times(0,T)$ was replaced by the Neummann
boundary condition 
$u_\nu=0$ on $\partial\Omega\times(0,T)$ in the case of ${\mathcal F}_2$.
The assumption $\alpha>0$ in the case of ${\mathcal F}_1$ was removed in \cite{R06}.

Assuming
\begin{equation} \label{assIanni}
{\mathcal Au}=-\Delta u+u, \quad  N=1, \quad \Omega \ \hbox{ bounded}, \quad n=3, \quad p=2,
\end{equation}
and the boundedness of the corresponding energy, 
Ianni \cite{I12} used the arguments in \cite{CL84,Q99} in order to prove estimate \eqref{AE} 
for the nonlinearities
\begin{equation} \label{F3}
{\mathcal F}_3(u)(x)=\Bigl(\int_\Omega\frac{|u(y)|^p}{|x-y|^{n-2}}\,dy\Bigr)|u(x)|^{p-2}u(x),
\end{equation}
$${\mathcal F}_4(u)=|u|^{q-1}u-{\mathcal F}_3(u), \quad q\in[3,p_S),  $$
and
$${\mathcal F}_5(u)=|u|^{q-1}u+{\mathcal F}_3(u), \quad q\in(1,p_S),  $$
which appear in the Choquard, Schr\"odinger-Poisson and Schr\"odinger-Poisson-Slater problems. 
In her subsequent paper \cite{I13}, she used her estimates of solutions of the problem
\begin{equation} \label{pbmSPS}
 \left.
\begin{aligned} u_t-\Delta u+u&=|u|^{q-1}u-
\lambda\Bigl(\int_\Omega\frac{|u(y)|^2}{|x-y|}\,dy\Bigr)u \ &&\hbox{ in }\Omega\times(0,T), \\
                u&= 0 \ &&\hbox{ on }\partial\Omega\times(0,T), \\
                u(\cdot,0)&= u_0 \ &&\hbox{ in }\Omega,
  \end{aligned} \ \right\}
\end{equation}
with $\lambda>0$ and $q\in[3,5)$ in order to prove 
the existence of infinitely many radial sign-changing steady states 
provided $\Omega$ is a ball or the whole of $\R^3$.
If $q\in(2,3)$, then
she did not obtain any priori estimates of solutions of \eqref{pbmSPS}, see \cite[Remark~4.3(ii)]{I12}.
On the other hand,
the existence of infinitely many radial sign-changing steady states of \eqref{pbmSPS}
with $\Omega=\R^3$
has recently been proved in \cite{GW20} by variational and approximation
arguments for any $q\in(2,5)$ (cf.~also \cite{LWZ16} for related results).
Assuming $q\in(2,3)$, we prove the missing a priori estimates of solutions to
\eqref{pbmSPS} with $\Omega$ bounded and starshaped. 
Our estimates are different from \eqref{AE}, but they still
guarantee the existence of infinitely many sign-changing radial steady states if $\Omega$ is a ball, 
and also the existence of positive and sign-changing steady states if $\Omega$ is a general starshaped bounded
domain (cf.\ related results in \cite{RS08}). 
In addition, denoting
$$\Phi_4(u):=  \frac1{q+1}\int_\Omega|u|^{q+1}\,dx-\frac\lambda4\int_\Omega\int_\Omega\frac{|u(x)|^2|u(y)|^2}{|x-y|}\,dy\,dx$$
and the energy
\begin{equation} \label{E-SPS} 
E(u):=\frac12\int_\Omega(|\nabla u|^2+u^2)\,dx-\Phi_4(u),
\end{equation}
our estimates guarantee that the nonincreasing function $t\mapsto E(u(\cdot,t))$
tends to $-\infty$ if the solution of \eqref{pbmSPS} becomes unbounded (in finite or infinite time).    
Finally, if $q\in[3,q_S)$, then we prove estimate \eqref{AE} without 
assuming the boundedness of the energy $E(u(\cdot,t))$.
More precisely, we have the following results.

\begin{theorem} \label{thm23}
Let $q\in(2,5)$, $\lambda>0$, $\Omega\subset\R^3$ be bounded and smooth, and let $u_0\in X$,
where $X=L_R(\Omega)$ and $R\in(\frac32(q-1)\max(1,\frac3q),\infty]$.
If $q\le3$, then assume also that $\Omega$ is starshaped.
Let $u$ be the solution of \eqref{pbmSPS} with the maximal existence time $T_{max}(u_0)$
and let $E$ be defined by \eqref{E-SPS}. 

(i) Assume
\begin{equation} \label{E-pos}
E(u(\cdot,t))\geq0 \ \hbox{ for }\ t\in(0,T_{max}(u_0)).
\end{equation}
Then 
\begin{equation} \label{AEpos}
T_{max}(u_0)=\infty \quad\hbox{and estimate }\ 
\eqref{AE2} \hbox{ is true.}
\end{equation}

(ii) Assume that \eqref{E-pos} fails. Then
\begin{equation} \label{E-neg}
\lim_{t\to T_{max}(u_0)}E(u(\cdot,t))=-\infty \ \hbox{ and }\
\lim_{t\to T_{max}(u_0)}\|u(\cdot,t)\|_X=\infty.
\end{equation}
In addition, if $q\ge3$, then 
\begin{equation} \label{TfiniteAE}
\hbox{$T_{max}(u_0)<\infty$ and estimate \eqref{AE} is true.}
\end{equation}
\end{theorem}

\begin{corollary} \label{cor23}
Let $q,\lambda$ and $\Omega$ be as in Theorem~\ref{thm23}.
Then problem \eqref{pbmSPS} possesses both positive and sign-changing
steady states. If $\Omega$ is a ball in $\R^3$, 
then \eqref{pbmSPS} possesses infinitely many radial sign-changing solutions.
\end{corollary}

\begin{remark} \rm
The nonlinear terms in problem \eqref{pbmSPS} have opposite sign.
If $q\in(2,3)$, 
then it is not easy to compare the corresponding terms in the potential $\Phi_4$
and --- unlike in the case $q\geq3$ --- we are not able to prove blow up in finite time
if the energy becomes negative.
The difference between the cases $q<3$ and $q>3$ can be explained as
follows:
Set $X:=H^1_0(\Omega)$.
If $q\in(3,p_S)$ and $v\in X\setminus\{0\}$, then 
\begin{equation} \label{Eray}
\lim_{t\to\infty} E(tv)=-\infty.
\end{equation}
If $q\in(2,3)$, then \eqref{Eray} fails but one still has
$\inf E:=\inf\limits_{v\in X}E(v)=-\infty$.
Notice also that in the scaling invariant case $q=2$ one has
\begin{equation} \label{lambdastar}
 \lambda^*:=\sup\{\lambda>0:\inf E=-\infty\} \in(0,\infty).
\end{equation}

Similar phenomena appear also in other problems with competing nonlinearities.
For example, consider the problem
\begin{equation} \label{eqCFQ}
\begin{aligned}
    u_t &= \Delta u-\lambda|u|^{p-1}u,    &\qquad& \hbox{in } \Omega\times(0,T),\\
    u_\nu  &= |u|^{r-1}u,  &\qquad& \hbox{on } \partial\Omega\times(0,T),\\
    u(\cdot,0) &=u_0,       &\qquad& \hbox{in }\Omega,
  \end{aligned}
\end{equation}
where $\Omega\subset\R^n$ is bounded, $u_\nu$ denotes the normal derivative of $u$,
$\lambda>0$ and $p\in(1,p_S)$, $r\in(1,r_S)$, $r_S:=\frac{n}{(n-2)_+}$.
In this case we set $X=H^1(\Omega)$ and 
$$ E(v)=\frac12\int_\Omega|\nabla v|^2\,dx+\frac\lambda{p+1}\int_\Omega|v|^{p+1}\,dx
 -\frac1{r+1}\int_{\partial\Omega}|u|^{r+1}\,dS. $$
It is known that if 
\begin{equation} \label{CQpq}
p<2r-1\quad\hbox{or}\quad (p=2r-1\ \hbox{ and }\ \lambda<r), 
\end{equation}
then this problem possesses solutions which blow up in finite time (and $\inf E=-\infty$),
while all solutions are global and bounded if $p>2r-1$ or $p=2r-1$ and $\lambda>r$, 
see \cite{CFQ91,RBT01,AMTR02}.
Condition \eqref{CQpq} also guarantees estimate \eqref{AE} if $n=1$
(or $u\geq0$ and $p<2r-1$), 
but if $n>1$, then estimate \eqref{AE} for sign-changing solutions is only known in the ``easy'' case $p<r$
corresponding to \eqref{Eray}, 
see \cite[Theorem~1.4 and Remark~3.8]{CQ04}, \cite{Q24} and the arguments in \cite{PQS07}.
Let us also mention that if $n>1$ and \eqref{CQpq} is true, 
then finite time blow-up for initial data with negative energy
is also known only if $p\le r$, 
while the proofs of blow-up in the case $p>r$
are based on the comparison principle
(which is not available in the case of \eqref{pbmSPS})
and require various restrictions on the initial data.

Next consider problem \eqref{eqCFQ}
in the scaling invariant case $p=2r-1$, and assume $n=1$.
Assume also $\lambda=r$ (notice that $r=\lambda^*$, where $\lambda^*$ is
defined in \eqref{lambdastar}).
Then all positive solutions of \eqref{eqCFQ} are global and unbounded (i.e.~\eqref{AEG} fails),
and they tend to a singular steady state as $t\to\infty$,
see \cite{CFQ91,FVW05}.
If we consider \eqref{pbmSPS}, then
similar phenomena can also be expected in the critical case $q=2$ and
$\lambda=\lambda^*$.
To have full scale-invariance as in \eqref{eqCFQ}, one should 
remove the linear term $u$ on the LHS of \eqref{pbmSPS}, i.e.
consider the equation
$$  u_t-\Delta u=|u|u-\lambda^*\Bigl(\int_\Omega\frac{|u(y)|^2}{|x-y|}\,dy\Bigr)u. $$
\qed
\end{remark}

In the case of nonlinearities ${\mathcal F}_4,{\mathcal F}_5$ and $q>3$, Ianni \cite{I12}
obtained her estimates by using the bootstrap argument in \cite{Q99}.
That bootstrap argument --- roughly speaking --- requires
a rather precise estimate of a suitable $L_r$-norm of ${\mathcal F}(u)$
by a suitable power of the potential $\Phi(u)$ (see the bootstrap condition \eqref{BC} below).
Such estimate is available for the nonlinearity ${\mathcal F}(u)=|u|^{q-1}u$ with $r=(q+1)/q$
since then $\int_\Omega|{\mathcal F}(u)|^r\,dx=C\Phi(u)$ for some constant $C$,
and the nonlinearities ${\mathcal F}_4,{\mathcal F}_5$ with $q>3$, $n=3$ and $p=2$ 
can be considered as perturbations of the nonlinearity ${\mathcal F}(u)=|u|^{q-1}u$
(cf.~also similar perturbation assumption \eqref{growthQ03}).
On the other hand, such simple perturbation argument is not available in the case of the nonlinearity ${\mathcal F}_3$
and, in fact, Ianni \cite{I12} did not use that bootstrap argument in this case:
Since $p=2<p_{CL}$ if $n=3$, it was sufficient for her to use the arguments from \cite{CL84}.     

Notice that the estimates in \cite{I12} remain true if we replace the assumption ${\mathcal Au}=-\Delta u+u$
with ${\mathcal Au}=-\Delta u$.  
Assuming \eqref{assNP1}, problem \eqref{AF} with
nonlinearity ${\mathcal F}_3$ (and general $n,p$) has been considered by several other authors,
but they either do not prove estimate \eqref{AE2} 
(see \cite{LM14,LM15,LL17,ZGRY24}, for example)
or their assumptions are very restrictive:
For example,  \cite[Remark 5]{LL20} provides such estimate only if $p<p_F$ 
which even does not cover the special case $n=3$ and $p=2$ studied in \cite{I12}.
These facts were our motivation to study the validity of \eqref{AE2} and \eqref{AE} for the problem
\begin{equation} \label{pbmChoquard}
 \left.
\begin{aligned} u_t-\Delta u&=
\Bigl(\int_\Omega\frac{|u(y)|^p}{|x-y|^{n-2}}\,dy\Bigr)|u|^{p-2}u \ &&\hbox{ in }\Omega\times(0,T), \\
                u&= 0 \ &&\hbox{ on }\partial\Omega\times(0,T), \\
                u(\cdot,0)&= u_0 \ &&\hbox{ in }\Omega.
  \end{aligned} \ \right\}
\end{equation}
The corresponding energy functional is given by
\begin{equation} \label{E-Choquard}
 E(u)=\frac12\int_\Omega|\nabla u|^2\,dx-\Phi_3(u),
\end{equation}
where
$$\Phi_3(u):= \frac1{2p}\int_\Omega{\mathcal F}_3(u)u\,dx = \frac1{2p}\int_\Omega\int_\Omega\frac {|u(x)|^p|u(y)|^p}{|x-y|^{n-2}}\,dy\,dx.$$ 
In order to avoid technical problems with non-Lipschitz nonlinearities, we assume $p\ge2$,
and since $p_S\le2$ if $n\ge6$, we also assume $n\leq5$.
Finally, if $n=2$, then \eqref{pbmChoquard} is of different nature
and estimate \eqref{AE} can be easily obtained for any $p>1$ (see the arguments in \cite{R06}),
hence we also assume $n\ge3$.
We have the following theorem:

\begin{theorem} \label{thmChoquard}
Assume that $\Omega$ is bounded and smooth, $3\le n\le5$ and $p\in[2,p^*)$, where
$$p^*:=p_S-\frac1{n-2}\,\frac8{2(n-2)\sqrt{n(n+3)}+2n^2-n-4}. $$
Let $u$ be a 
solution of \eqref{pbmChoquard} with $u_0\in X$,
where $X=L_q(\Omega)$ and $q\in(\frac n2(p-1),\infty]$.
Then estimate \eqref{AE} is true.
Estimate \eqref{AE} is also true with $\delta=0$ and $T_{max}(u_0)<\infty$ 
provided the energy $E(u(\cdot,t))$
remains bounded for $t\in (0,T_{max}(u_0))$.

Consequently, we have the following:

(i) If $u$ blows up at $T<\infty$, then $E(u(t))\to-\infty$ as $t\to T-$.

(ii) The maximal existence time $T_{max}:X\to(0,\infty]:u_0\mapsto T_{max}(u_0)$ is continuous.

(iii) Let ${\mathcal G}:=\{u_0\in X:T_{max}(u_0)=\infty\}$. Then ${\mathcal G}$ is closed.
If $u_0\in\partial{\mathcal G}$, then the $\omega$-limit set $\omega(u_0)$ of solution $u$ of \eqref{pbmChoquard}
consists of nontrivial steady states.
\end{theorem}

Our result does not cover the full subcritical range, but our condition $p<p^*$
is not far from the conjectured optimal condition $p<p_S$
(and is significantly better than the condition $p<p_{CL}$). 
Table~\ref{tab1} compares values $p_S,p^*,p_{CL}$ and the assumptions
in \cite{I12} ($p=2$, $n=3$) and \cite{LL20} ($p<p_F$) for $n\in\{3,4,5\}$. 

\def\vb{\hbox{\vbox to5truemm{\vss}}}
\renewcommand{\arraystretch}{1} 
\begin{table}[ht]
\begin{center}
\begin{tabular}{ |c||c|c|c|c|c|  }
\hline 
 \vb $n$& $p_S$ & $p^*$ & $p_{CL}$ & \cite{I12} & $p_F$ \cite{LL20} \\  [0.5ex]
\hline
 \vb 3 & 5 & $>4.589$ & 3.4 & $p=2$ & $1.\overline{6}$  \\  [0.5ex] 
\hline
 \vb 4 & 3 & $>2.911$ & 2.5 & & $1.5$ \\  [0.5ex]
\hline
 \vb 5 & $2.\overline{3}$ & $>2.299$ & $2.\overline{09}$ & & $1.4$ \\  [0.5ex]
\hline
\end{tabular}
  \end{center}
\kern2mm 
\caption{Conjectured and sufficient upper bounds on $p$ guaranteeing \hskip0mm plus10mm 
\hbox{estimate~\eqref{AE2}} for problem \eqref{pbmChoquard}.}
\label{tab1}
\end{table}

Similarly as in \cite{Q04,I13},
assertion (iii) in Theorem~\ref{thmChoquard} combined with some topological arguments
can also be used in order to prove the existence of sign-changing steady states
(or infinitely many radial sign-changing steady states if $\Omega$ is a ball),
cf.~Corollary~\ref{cor23} above.

A stronger version of a priori estimates \eqref{AE}
can often be obtained by scaling and doubling arguments provided
the corresponding parabolic Liouville type theorem is known, 
see \cite{Q21,Q22,Q22a,QS24} and the references therein
and see also Remark~\ref{remFwd}.
However, this approach does not seem to be suitable for problems
with ${\mathcal A}=(-\Delta)^\alpha$, $\alpha\in(0,1)$: 
For example, even in the case of the model nonlinearity
${\mathcal F}(u)=|u|^{p-1}u$, the corresponding
parabolic Liouville type theorem is known only under the very restrictive condition $p\leq1+\frac{2\alpha}n$,
see \cite{S75,DN23}. 
If $\Omega=\R^n$, ${\mathcal A}=(-\Delta)^\alpha$, $\alpha\in(0,1)$ and ${\mathcal F}(u)=|u|^{q-1}u-u$, for example,
then one can use the arguments in \cite{Q03}
to obtain estimates \eqref{AE} in the subcritical case, 
since the required  maximal $L_p$-regularity is well known
and the domain of ${\mathcal A}$ is just an analogue of the domain in the case $\alpha=1$.
If $\Omega$ is bounded, ${\mathcal A}=(-\Delta)^\alpha$ with $\alpha\in(0,1)$ and we consider the Dirichlet problem, 
then the required maximal regularity has also been proved (see \cite{G15,AG24} and the references therein),
but the domain of the operator $(-\Delta)^\alpha$ with the Dirichlet conditions
is not a simple analogue of the domain in the case $\alpha=1$. Fortunately,
this fact does not have a strong influence on the proof in \cite{Q99}
so that we can prove estimate \eqref{AE} for solutions of the problem 
\begin{equation} \label{fracLap}
 \left.
\begin{aligned} u_t+(-\Delta)^\alpha u&=|u|^{p-1}u \ &&\hbox{ in }\Omega\times(0,T) \\
                 u&= 0 \ &&\hbox{ in }(\R^n\setminus\Omega)\times(0,T) \\
                u(\cdot,0)&= u_0 &&\hbox{ in }\Omega
  \end{aligned} \ \right\}
\end{equation}
in the full subcritical range $p\in(1,p_S(\alpha))$, where
$$ p_S(\alpha):=\frac{n+2\alpha}{n-2\alpha}.$$
In this case, the energy is defined by
\begin{equation} \label{E-frac}
 E(u)=\frac12\int_\Omega|(-\Delta)^{\alpha/2}u|^2\,dx-\frac1{p+1}\int_\Omega|u|^{p+1}\,dx.
\end{equation}

\begin{theorem} \label{thmFracLap}
Let $\Omega\subset\R^n$ be bounded and smooth, $\alpha\in(0,1)$, $p\in(1,p_S(\alpha))$ and 
$u_0\in X:=L_q(\Omega)$, where $q\in(\max(1,\frac{n}{2\alpha}(p-1)),\infty]$.
Let $u$ be a solution of \eqref{fracLap}.
Then estimate \eqref{AE} is true.
\end{theorem}

The consequences of estimate \eqref{AE} in Theorem~\ref{thmChoquard}
are also true for problem \eqref{fracLap}.     

We could also consider other nonlocal operators, for example
$${\mathcal A}(u)=-a\Bigl(\int_\Omega|\nabla u|^2\Bigr)\Delta u,\quad
\hbox{where 
$a(\cdot)\geq a_0>0$ is continuous, $\lim\limits_{s\to\infty}a(s)=a_\infty$.}$$
However, since the proofs would not involve important new ideas,
we refrain from doing so.
Similarly, our arguments can be straightforwardly applied to various parabolic systems. 

\begin{remark} \label{rempS} \rm {\bf The role of critical exponents.}
To see the role of the critical exponents $p_S$ and $q^*:=\frac n2(p-1)$ in our estimates
and their consequences, consider the model scalar problem \eqref{MP},
where $\Omega\subset\R^n$ is a ball and $p>1$.

If $X:=L^\infty(\Omega)$ and we consider global
solutions, then estimate \eqref{AE2} is true if and only if $p<p_S$,
while the weaker estimate \eqref{AEG} is true if and only if $p\ne p_S$, see
\cite[Theorems~22.1, 22.4*, 28.7*]{QS19}.
The condition $p<p_S$ is necessary not only for estimates \eqref{AE} and
\eqref{AE2},
but also for most of their consequences 
(blow-up of energy, continuity of the maximal existence time,
existence of positive steady states etc).
More precisely, if $p\ge p_S$, then problem \eqref{MP} with $\Omega$ being
a ball does not possess positive steady states,
and if $p>p_S$, then there exist radially symmetric positive initial data
$u_0$ such that $T_{max}(u_0)<\infty$, $E(u(t))\ge0$ for
$t\in(0,T_{max}(u_0))$ and $T_{max}$ is not continuous at $u_0$.
In addition, if $n\ge11$ and $p$ is large enough, then 
the blow-up rate estimate \eqref{rateGK} also fails.
Since estimate \eqref{AEG} is true for all $p>p_S$,
it cannot imply the consequences listed above.

If $q^*>1$, then the condition $q\ge q*$ is sufficient and necessary for the
assertion
\begin{equation} \label{limLR}
\lim_{t\to T_{max}(u_0)}\|u(\cdot,t)\|_q=\infty \ \hbox{ whenever }\ T_{max}(u_0)<\infty,
\end{equation}
see \cite{FML85,MS19},
and see also \cite{MT24,MT23} for recent results on \eqref{limLR} if $q=q^*$
and $\Omega$ is more general.
Some $L_q$ (or $L_{q,loc}$) norms of $u(\cdot,t)$ can stay bounded
as $t\to T_{max}(u_0)<\infty$ also in the case of nonlocal problems: See  
\cite[Theorems~44.2 and~44.6]{QS19} or \cite{B02,AZ21,DKZ21}, for example. 
\qed\end{remark}

\section{Proofs}

In the whole section we implicitly assume that
\begin{equation} \label{assOmega}
\hbox{$\Omega\subset\R^n$ is bounded and smooth}
\end{equation}
and by $\|\cdot\|_r$ or $\|\cdot\|_{s,r}$ we denote the norm in the Lebesgue
space $L_r(\Omega)$
or the Bessel potential space $H^s_r(\Omega)$ respectively.
In particular, $\|u\|_{1,2}^2=\int_\Omega(|\nabla u|^2+u^2)\,dx$.
We will also use the notation $u(t):=u(\cdot,t)$ and $u(t;u_0):=u(\cdot,t;u_0)$.
If a constant $C$ depends on some norm of the solution $u$,
then we implicitly assume that this dependence is nondecreasing.
In particular, \eqref{AE2} will be written in the form
$$ 
\|u(t;u_0)\|_X \leq C(\|u_0\|_X), \quad t\in[0,\infty).
$$

\subsection{Auxiliary results}
In this subsection we study well-posedness of problems
\eqref{pbmSPS}, \eqref{pbmChoquard} and \eqref{fracLap}
in the Lebesgue spaces. These results will play an important role
in the proofs of our a priori estimates.

\begin{lemma} \label{lem1}
Let $n\ge3$ and $p\ge2$. Given $\alpha>1$ and $r>1$, define $m=m(r\alpha,n)$ by
\begin{equation} \label{eR3b}
 \frac1m=\frac1{r\alpha}+\frac2n,
\end{equation}
and let $\alpha':=\frac{\alpha}{\alpha-1}$ denote the dual exponent to $\alpha$.

(i) If $n=3$, $p\in[2,3)$ and $r>\frac n2(p-1)$, then there exist $\alpha>1$ and $R>r$ such that
\begin{equation} \label{rR3a}
 m>1,\quad R\geq\max(pm,(p-1)r\alpha')\quad\hbox{and}\quad \beta:=\frac n2\Bigl(\frac1r-\frac1R\Bigr)<\frac1{2p-1}.
\end{equation}

(ii) Let either 
\begin{equation} \label{rR1}
 r\in(1,\frac n2\frac{p}{p-1}),\quad \frac1r<2-\frac2n-\frac1p,
 \quad R=\frac{2p-1}{\frac1r+\frac{2}n},
\end{equation}
or
\begin{equation} \label{rR3}
 n=3, \quad p\in[2,3), \quad r>\frac n2(p-1),\quad \alpha>1,\quad 
 \hbox{$m,R$ satisfy \eqref{rR3a}.}
\end{equation}
Then ${\mathcal F}_3:L_R(\Omega)\to L_r(\Omega)$ satisfies
\begin{equation} \label{estF3}
\|{\mathcal F}_3(u)-{\mathcal F}_3(v)\|_r \le C(\|u\|_R^{2(p-1)}+\|v\|_R^{2(p-1)})\|u-v\|_R.
\end{equation} 
\end{lemma}

\begin{proof}
(i)
Since $r>\frac n2(p-1)$, we can
choose $\alpha>1$ such that
$$ 2-p+\frac{2r}n\frac{p-1}{2p-1}>\frac1\alpha>p-\frac rn\frac{6p-2}{2p-1} $$
and then choose $z>1$ such that
$$ \min\Bigl(\frac1{(p-1)\alpha'},\frac1p\Bigl(\frac1\alpha+\frac{2r}n\Bigr)\Bigr)\ge\frac1z>1-\frac{2r}{n(2p-1)}.$$
Set $R:=zr$. Then \eqref{rR3a} is true.

(ii) 
If \eqref{rR1} is true, then set
$$ \alpha:=\frac{2p-1}{p-(p-1)\frac{2r}n} $$ 
and notice that $m>1$ and $R=mp=(p-1)r\alpha'$.

Let $u\in L_R(\Omega)$ and $w(x):=\int_\Omega\frac{|u(y)|^p}{|x-y|^{n-2}}\,dy$. 
Then using the H\"older and Hardy-Littlewood-Sobolev inequalities
(see \cite{LL01}) we obtain
$$ \begin{aligned}
  \|{\mathcal F}_3(u)\|_r &= \|w|u|^{p-2}u\|_r 
     \leq \|w\|_{r\alpha}\||u|^{p-1}\|_{r\alpha'} \\
   &\leq C\||u|^p\|_m\||u|^{p-1}\|_{r\alpha'} \\
 &= C\|u\|_{mp}^p\|u\|_{r\alpha'(p-1)}^{p-1} \leq C\|u\|_R^{2p-1},
\end{aligned}$$
hence ${\mathcal F}_3:L_R(\Omega)\to L_r(\Omega)$ is bounded on bounded sets.
Analogous estimates imply \eqref{estF3}. 
\end{proof}

\begin{proposition} \label{prop2}
Let $n\ge3$, $p\ge2$, $R>\frac n2(p-1)$. 
Let $u_0\in X:=L_R(\Omega)$ and let $E$ be defined by \eqref{E-Choquard}.
Then there exists a unique local mild solution of \eqref{pbmChoquard}.
This solution is classical for $t>0$ and the function $(0,T)\to\R:t\mapsto E(u(\cdot,t))$ is $C^1$
with $\frac{d}{dt}E(u(\cdot,t))=-\int_\Omega u_t^2(x,t)\,dx$.
In addition, 
there exist $\delta=\delta(\|u_0\|_X)>0$ and $C=C(\|u_0\|_X)$ such that
$T_{max}(u_0)>\delta$, $\|u(\cdot,t)\|_X\leq C$ for $t\in[0,\delta]$ and 
$\|u(\cdot,\delta)\|_{2(1-\eps),r}\leq C$ for any $\eps>0$ and
$r<\frac{n}2\frac{p}{p-1}$. In particular,
$\|u(\cdot,\delta)\|_\infty+\|u(\cdot,\delta)\|_{1,2}\leq C$.
Given $\eps>0$, we can choose $\delta\in(0,\eps)$, but then the constant $C$ also depends on $\eps$.
Finally, if $T_{max}(u_0)<\infty$, then $\lim_{t\to T_{max}(u_0)}\|u(\cdot,t)\|_X=\infty$.
\end{proposition}

\begin{proof}
We will modify the proofs of \cite[Theorem~15.2]{QS19} and some related results.

(i)
First assume either $n\ge4$ or $p\ge3$, and let $R>\frac n2(p-1)$.
Notice that these assumptions guarantee $R>p$. 
Consider $T>0$ small and let $Y_T:=C([0,T],L_R(\Omega))$.
Fix $K>0$ and assume $\|u_0\|_R\le K$.
Choose $M>K$ and let $B_M=B_{M,T}$ denote the closed ball in $Y_T$ with center 0 and radius $M$.
By $e^{-tA}$ we denote the Dirichlet heat semigroup in $\Omega$.
We will use the Banach fixed point theorem for the mapping
$\Phi_{u_0}:B_M\to B_M$, where
\begin{equation} \label{Phiu0}
 \Phi_{u_0}(u)(t):=e^{-tA}u_0+\int_0^t e^{-(t-s)A}{\mathcal F}_3(u(s))\,ds.
\end{equation}
If $R<\frac n2 p$, 
then set $\tilde R:=R$;
otherwise set $\tilde R:=\frac n2 p-\eps$, where $\eps>0$ is small.
Let  $r$ be defined by the last formula in \eqref{rR1} with $R$ replaced by $\tilde R$, 
i.e.~$\frac1r=\frac{2p-1}{\strut\tilde R}-\frac2n$.
Then \eqref{rR1} with $R$ replaced by $\tilde R$ is true
and $L_R(\Omega)\hookrightarrow L_{\tilde R}(\Omega)$, hence
${\mathcal F}_3:L_R(\Omega)\to L_r(\Omega)$ satisfies \eqref{estF3}.
In addition, $\beta:=\frac n2(\frac1r-\frac1R)<1$. 

Using $L_p$-$L_q$-estimates (see \cite[Proposition 48.4*]{QS19}) and \eqref{estF3},
for any $u,v\in B_M$ with $u(0)=u_0$ and $v(0)=v_0$ we obtain
$$\begin{aligned}
\|\Phi_{u_0}&(u)(t)-\Phi_{v_0}(v)(t)\|_R \\
  &\le \|e^{-tA}(u_0-v_0)\|_R
   +\int_0^t \|e^{-(t-s)A}
     \bigl({\mathcal F}_3(u(s))-{\mathcal F}_3(v(s))\bigr)\|_R\,ds \\
  &\le \|u_0-v_0\|_R+\int_0^t (t-s)^{-\beta}\|
     {\mathcal F}_3(u(s))-{\mathcal F}_3(v(s))\|_r\,ds \\
 &\le \|u_0-v_0\|_R+C\int_0^t (t-s)^{-\beta}
   (\|u(s)\|_R^{2(p-1)}+\|v(s)\|_R^{2(p-1)})\|u(s)-v(s)\|_R\,ds \\
 &\le \|u_0-v_0\|_R+ CM^{2(p-1)}T^{1-\beta}\|u-v\|_{Y_T},
\end{aligned}$$
hence $\Phi_{u_0}$ is a contraction if $T=T(M)$ is small enough.
In addition, choosing $v=0$ we obtain
$$ \|\Phi_{u_0}(u)(t)\|_R\leq \|u_0\|_R+ CM^{2(p-1)}T^{1-\beta}\|u\|_{Y_T}, $$
hence $\Phi_{u_0}$ maps $B_M$ into $B_M$.
Consequently, $\Phi_{u_0}$ possesses a unique fixed point in $B_M$ which corresponds
to the mild solution of \eqref{pbmChoquard}.
Now given $t\in(0,T]$ and $\eps>0$ small we have the estimate
$$ \begin{aligned}
\|u(t)\|_{2-2\eps,r} &\leq t^{-1+\eps}\|u_0\|_r+\int_0^t(t-s)^{-1+\eps}\|{\mathcal F}_3(u(s))\|_r\,ds \\
 &\leq t^{-1+\eps}\|u_0\|_r+C\int_0^t(t-s)^{-1+\eps}\|u(s)\|_R^{2p-1}\,ds \\
 &\leq t^{-1+\eps}\|u_0\|_r+CM^{2p-1}T^\eps.
\end{aligned}$$
If $\eps$ is small enough, then
$H^{2-2\eps}_r(\Omega)\hookrightarrow L_{R_1}(\Omega)$,
where $R_1=\infty$ if $R\geq\frac n2 p$ and $R_1>R$ if $R<\frac n2 p$.
Therefore
we can use a bootstrap argument to conclude that, given $\delta\in(0,T)$, $u(t)$ is bounded for $t\in[\delta,T]$ in the space
$\hat X:=H^{2-2\eps}_{\hat r}(\Omega)$ for any $\hat r<\frac n2\frac{p}{p-1}$ and $\eps>0$.
Choosing $\hat r$ close to its upper bound and $\eps$ close to 0,
$\hat X$ is embedded both in $H^1_2(\Omega)$ and $C^\kappa(\bar\Omega)$ for some $\kappa>0$.
In addition, decreasing $\kappa$ if necessary, \cite[Theorem~51.7]{QS19} implies 
$u\in C^\kappa([\delta,T],C^\kappa(\bar\Omega))$ and the Schauder estimates imply
that $u$ is a classical solution for $t>0$.
In addition, these estimates also imply that $u$ can be continued to the maximal solution
and $\lim_{t\to T_{max}(u_0)}\|u(\cdot,t)\|_R=\infty$
whenever $T_{max}(u_0)<\infty$,
and the arguments in \cite[Example 51.28]{QS19} also show the claimed properties
of the energy functional. 

(ii)
Next assume $n=3$, $p\in[2,3)$ and $R>\frac n2(p-1)$.
Lemma~\ref{lem1}
with $(r,R)$ replaced by $(R,zR)$ guarantees the existence of $z>1$ such that estimate
\eqref{estF3} with $(r,R)$ replaced by $(R,zR)$ is true and
$\beta=\frac n{2R}\frac{z-1}z<\frac1{2p-1}$.
Now consider the Banach space
$$Y_T:=\{u\in L_{\infty,loc}\bigl((0,T),L_{zR}(\Omega)\bigr): \|u\|_{Y_T}<\infty\},
  \qquad \|u\|_{Y_T}:=\sup_{0<t<T}t^\beta\|u(t)\|_{zR}$$
and assume $\|u_0\|_R\leq K$.
We will again use the Banach fixed point theorem for the mapping $\Phi_{u_0}:B_M\to B_M$ with $M>K$.
Assume $u,v\in B_M$ and $v_0\in L_R(\Omega)$. Then
$$\begin{aligned}
  t^\beta &\|\Phi_{u_0}(u)(t)-\Phi_{v_0}(v)(t)\|_{zR} \\
  &\leq t^\beta\|e^{-tA}(u_0-v_0)\|_{zR}
   +t^\beta\int_0^t \|e^{-(t-s)A}
     \bigl({\mathcal F}_3(u(s))-{\mathcal F}_3(v(s))\bigr)\|_{zR}\,ds \\
  &\leq \|u_0-v_0\|_{R}
   +t^\beta\int_0^t (t-s)^{-\beta}
   \|{\mathcal F}_3(u(s))-{\mathcal F}_3(v(s))\|_{R}\,ds \\
  &\leq \|u_0-v_0\|_{R}
   +t^\beta\int_0^t (t-s)^{-\beta}
    \bigl(\|u(s)\|_{zR}^{2(p-1)}+\|v(s)\|_{zR}^{2(p-1)}\bigr)\|u(s)-v(s)\|_{zR}\,ds \\
  &\leq \|u_0-v_0\|_{R}
   +CM^{2(p-1)}t^\beta\int_0^t (t-s)^{-\beta}s^{-(2p-1)\beta}
     \|u-v\|_{Y_T}\,ds \\
 &\leq \|u_0-v_0\|_{R}
   +CM^{2(p-1)}T^{1-\beta(2p-1)}\|u-v\|_{Y_T}.
\end{aligned} $$
In particular, choosing $v_0=0$ and $v=0$ in 
we have
$$
 \|\Phi_{u_0}(u)\|_{Y_T}
\leq \|u_0\|_R+CM^{2(p-1)}T^{1-(2p-1)\beta}\|u\|_{Y_T}
$$
and we conclude as above.
\end{proof}

\begin{remark} \label{rem-WP} \rm
(i) Condition $R>\frac n2(p-1)$ in Proposition~\ref{prop2} is the same
as the optimal condition for the well-posedness in $L_R(\Omega)$ in the case
of the nonlinearity ${\mathcal F}(u)=|u|^{p-1}u$ (see \cite{QS19}).
In the proof of Theorem~\ref{thmChoquard} we prove estimates of $\|u(t)\|_R$
with $R<p+1$. Since $p+1>\frac n2(p-1)$ if $p<p_S$, such
estimates and our well-posedness result will guarantee \eqref{AE}.
The well-posedness of problem \eqref{pbmChoquard} 
in $L_R$-spaces was also studied in \cite{LM14} 
but this paper requires stronger condition
$R>\frac n2(p-1)(2-\frac1p)$ and $R\ge n-1$.

(ii) Similar estimates as in the proof of Proposition~\ref{prop2} show that
the solution depends continuously on the initial data.
More precisely, let $u$ and $v$ be solutions of \eqref{pbmChoquard}
with initial data $u_0$ and $v_0$, respectively.
Then there exists $\delta=\delta(\|u_0\|_X,\|v_0\|_X)>0$ such that  
$\|u(t)-v(t)\|_X\leq2\|u_0-v_0\|_X$ for $t\in[0,\delta]$.
In addition, if $T<\min(T_{max}(u_0),T_{max}(v_0))$, 
$\|u(t)\|_X+\|v(t)\|_X\leq M$ for $t\in[0,T]$, and $\delta\in(0,T)$, then there exists
$C=C(M,\delta,T)$ such that
$\|u(t)-v(t)\|_\infty+\|u(t)-v(t)\|_{1,2}\leq C\|u_0-v_0\|_X$ for
$t\in[\delta,T]$. In particular,
 $|E(u(t))-E(v(t))|\leq C\|u_0-v_0\|_X$ for $t\in[\delta,T]$.
See \cite[Step~3 in the proof of Theorem~15.2 and Theorem~51.7]{QS19}
for related results.

(iii) The function $u\equiv0$ is an asymptotically stable
steady state of \eqref{pbmChoquard}. Consequently,
its domain of attraction 
$${\mathcal D}_0:=\{u_0\in X:T_{max}(u_0)=\infty
\ \hbox{ and }\ u(t;u_0)\to 0\hbox{ as }t\to\infty\}$$
is an open nonempty set.
If $u_0\in\partial{\mathcal D}_0$, then (ii) implies $E(u(t))\geq0$ for
$t\in[0,T_{max}(u_0))$.
\qed
\end{remark}

\begin{proposition} \label{prop3}
Let $n=3$, $p=2$, $q\in(2,5)$ and $R>\frac32(q-1)\max(1,\frac3q)$. Let $u_0\in X:=L_R(\Omega)$
and let $E$ be defined by \eqref{E-SPS}.
Then there exists a unique local mild solution of \eqref{pbmSPS}. 
This solution is classical for $t>0$ and the function $(0,T)\to\R:t\mapsto E(u(\cdot,t))$ is $C^1$
with $\frac{d}{dt}E(u(\cdot,t))=-\int_\Omega u_t^2(x,t)\,dx$. 
In addition, there exist $\delta=\delta(\|u_0\|_X)>0$ and $C=C(\|u_0\|_X)$ such that
$T_{max}(u_0)>\delta$, $\|u(\cdot,t)\|_X\leq C$ for $t\in[0,\delta]$ and 
$\|u(\cdot,\delta)\|_\infty+\|u(\cdot,\delta)\|_{1,2}\leq C$.
Given $\eps>0$, we can choose $\delta\in(0,\eps)$, but then the constant $C$ also depends on $\eps$.
Finally, if $T_{max}(u_0)<\infty$, then $\lim_{t\to T_{max}(u_0)}\|u(\cdot,t)\|_X=\infty$.
\end{proposition}

\begin{proof}
First assume $q\in(2,3)$.
We will proceed as in Case (ii) of the proof of Proposition~\ref{prop2}.
Notice that if we use Lemma~\ref{lem1} with $(r,R)$ replaced by $(R,zR)$,
then we may assume $z\geq q$, due to our assumption $R>\frac92\frac{q-1}q$.
Hence we have 
$$\begin{aligned}
\|{\mathcal F}_4(u)-{\mathcal F}_4(v)\|_R 
&\le C(\|u\|_{zR}^{2}+\|v\|_{zR}^{2}+\|u\|_{zR}^{q-1}+\|v\|_{zR}^{q-1})\|u-v\|_{zR} \\
&\le C(\|u\|_{zR}^{2}+\|v\|_{zR}^{2}+1)\|u-v\|_{zR},
\end{aligned}
$$
and the same arguments as in the proof of Proposition~\ref{prop2}
show the well-posedness in $L_R(\Omega)$
and also the remaining assertions in Proposition~\ref{prop3}.

If $q\ge3$, then the proof is a straightforward modification of the proof in
the case $\lambda=0$ in \cite{QS19}.
\end{proof}

\begin{remark} \label{rem-WP2} \rm
(i) 
The well-posedness in $L_R$-spaces for a modification of problem \eqref{pbmSPS}
(including the nonlinearity ${\mathcal F}_4$ with general $n,p$ and $q=2p-1<p_S$) 
was also studied in \cite{LM15}; 
this paper requires $R>\frac n2\max(q-1,1)$, hence $R>3$ if $n=3$, $p=2$ and $q=2p-1=3$.

(ii) The assertions in Remark~\ref{rem-WP}(ii) and (iii) remain true for
problem \eqref{pbmSPS}.
\qed
\end{remark}

\begin{proposition} \label{prop4}
Consider problem \eqref{fracLap} with $\alpha\in(0,1)$.
Let $u_0\in X:=L_q(\Omega)$, where $q>\max(1,\frac{n}{2\alpha}(p-1))$,
and let $E$ be defined by \eqref{E-frac}.
Then there exists a unique local mild solution of \eqref{fracLap}. 
This solution satisfies the equation in \eqref{fracLap} a.e.~and the function $(0,T)\to\R:t\mapsto E(u(\cdot,t))$ is 
locally Lipschitz continuous 
with $\frac{d}{dt}E(u(\cdot,t))=-\int_\Omega u_t^2(x,t)\,dx$ a.e. 
In addition, there exist $\delta=\delta(\|u_0\|_X)>0$ and $C=C(\|u_0\|_X)$ such that
$T_{max}(u_0)>\delta$, $\|u(\cdot,t)\|_X\leq C$ for $t\in[0,\delta]$ and 
$\|u(\cdot,\delta)\|_{\alpha,2}+\|u(\cdot,\delta)\|_\infty\leq C$.
Finally, if $T_{max}(u_0)<\infty$, then $\lim_{t\to T_{max}(u_0)}\|u(\cdot,t)\|_X=\infty$.
\end{proposition}

\begin{proof}
Let $A=A_q$ denote the Dirichlet fractional Laplacian $(-\Delta)^\alpha_D$ 
considered as an unbounded operator in $L_q(\Omega)$, where $q\in(1,\infty)$.
Then $A$ generates an analytic semigroup in $L_q(\Omega)$,
the domain of $A$ is the transmission space $D_q:=H^{\alpha,(2\alpha)}_q(\overline\Omega)$
and, given $J:=(0,T)$ with $T<\infty$ and $p\in(1,\infty)$, 
$$ ({\mathbb E}_0,{\mathbb E}_1):=\bigl(L_p(J,L_q(\Omega)),L_p(J,D_q)\cap W^1_p(J,L_q(\Omega))\bigr) $$
is a pair of maximal regularity for $A$, see \cite{G15,AG24}.
In particular,
if $f\in {\mathbb E}_0$ and $u_0$ belongs to the real interpolation space $X_{p,q}:=(L_q(\Omega),D_q)_{1-1/p,p}$,
then the solution of the problem
$$ 
\begin{aligned} u_t+(-\Delta)^\alpha u&=f \ &&\hbox{ in }\Omega\times(0,T), \\
                 u&= 0 \ &&\hbox{ in }(\R^n\setminus\Omega)\times(0,T), \\
                u(\cdot,0)&= u_0 &&\hbox{ in }\Omega
  \end{aligned} 
$$
satisfies
\begin{equation} \label{MRfrac}
 \|u\|_{{\mathbb E}_1} \leq C\cdot(\|f\|_{{\mathbb E}_0}+\|u_0\|_{X_{p,q}}). 
\end{equation}
Notice that estimate \eqref{MRfrac} is also true if $\alpha=1$ and $D_q=\{u\in H^2_q(\Omega):u=0\hbox{ on }\partial\Omega\}$.
Assume $2\alpha-\frac nq>-\frac nR$. Then there exist $r>q$ and $s>0$ such that
\begin{equation} \label{rs1}
 \alpha+s+\frac nR>\frac nr>\frac nq+s-\alpha\quad\hbox{and}\quad s<\min(\frac1r,\alpha).
\end{equation}
These inequalities and the definition of $H^{\alpha,(2\alpha)}_q(\overline\Omega)$ imply 
$H^{\alpha}_q(\Omega)\hookrightarrow H^s_r(\Omega)$ and
$$H^{\alpha,(2\alpha)}_q(\overline\Omega)\hookrightarrow H^{\alpha+s}_r(\Omega)\hookrightarrow L_R(\Omega).$$
Similarly, if $q>\frac{2n}{n+2\alpha}$, then we can find $r>q$ and $s>0$ such that
\begin{equation} \label{rs2}
 s+\frac n2>\frac nr>\frac nq+s-\alpha\quad\hbox{and}\quad s<\min(\frac1r,\alpha),
\end{equation}
hence 
$$ H^{\alpha,(2\alpha)}_q(\overline\Omega)\hookrightarrow H^{\alpha+s}_r(\Omega)
  \hookrightarrow H^\alpha_2(\Omega).$$
Next consider the interpolation space 
$X_\theta:=(L_q(\Omega),H^{\alpha,(2\alpha)}_q(\overline\Omega))_\theta$,
where $(\cdot,\cdot)_\theta$ is an interpolation functor of exponent $\theta$.
Similarly as above, if $2\alpha-\frac nq>-\frac nR$ or $q>\frac{2n}{n+2\alpha}$,
then there exists $\eps=\eps(n,\alpha,q,R)>0$ such that 
\begin{equation} \label{embHalpha}
X_\theta\hookrightarrow L_R(\Omega)\ \hbox{ or }\ X_\theta\hookrightarrow H^\alpha_2(\Omega),
\end{equation} 
respectively, provided $\theta\in(1-\eps,1)$.
In fact, the emeddings above and the interpolation results \cite[Theorems~VII.2.7.2 and VII.2.8.3]{A19}
imply
$$ X_\theta \hookrightarrow (L_q(\Omega),H^{\alpha+s}_r(\Omega))_\theta
 \hookrightarrow H^{\alpha+\tilde s}_{\tilde r}(\Omega) $$
with $\tilde s=s-\eps(s+\alpha)$ and $\tilde r=r/(1+\eps r/q)$,
and \eqref{rs1}, \eqref{rs2} remain true if we replace $r,s$ with $\tilde r,\tilde s$
provided $\eps$ is small enough.

Since the semigroup $e^{-tA}$ satisfies
$\|e^{-tA}u\|_{X_\theta}\leq Ct^{-\theta}\|u\|_q$,
the proof of local existence for problem \eqref{fracLap} 
and the estimates of the solution $u(\cdot,t)$ in $X=L_q(\Omega)$ with $q>\frac{n}{2\alpha}(p-1)$
are the same as in the proof of \cite[Theorem~15.2]{QS19}
which deals with the case $\alpha=1$:
It is sufficient to replace $\alpha:=\frac{n(p-1)}{2pq}$ in that proof 
with $\beta:=\frac{n(p-1)}{2\alpha pq}$.
The estimates in $H^\alpha_2(\Omega)$ and $L^\infty(\Omega)$ follow by the same bootstrap arguments as in \cite{QS19}.
In fact, these arguments also imply
\begin{equation} \label{estXpq}
\|u(\cdot,\delta)\|_{X_{p,q}}\leq C \quad\hbox{for any }\ p,q>1.
\end{equation}
If $T_{max}(u_0)<\infty$, then the property $\lim_{t\to T_{max}(u_0)}\|u(\cdot,t)\|_X=\infty$
follows from the well-posedness of the problem in $X$.
The solution $u$ satifies the equation in \eqref{fracLap} a.e.~due to \eqref{MRfrac}. 
If we fix $t_0>0$ such that $u(\cdot,t_0)\in X_1=D_q$, then $\|u(t)-u(t_0)\|_q\leq C(t-t_0)$ for $t>t_0$ close to $t_0$
since $\frac1t(e^{-tA}u_0-u_0)\to -Au_0$ in $X$.
The variation-of-constants formula and
Gronwall's inequality then imply $\|u(t+\tau)-u(t_0+\tau)\|_q\leq C(t-t_0)$ for $\tau\in(0,1]$,
and standard bootstrap arguments imply the local Lipschitz continuity of $u:(0,T)\to \tilde X_\theta$ for any $\theta<1$,
where $\tilde X=L_{\tilde q}(\Omega)$, $\tilde q\geq q$.
This fact and \eqref{embHalpha} guarantee the local Lipschitz continuity of $t\mapsto E(u(\cdot,t))$.
The a.e. differentiability of this function and estimate \eqref{MRfrac} guarantee
$\frac{d}{dt}E(u(\cdot,t))=-\int_\Omega u_t^2(x,t)\,dx$ a.e.
\end{proof}

\subsection{Proofs of Theorem~\ref{thm23} and Corollary~\ref{cor23}}

By $c,C$ we denote positive constants which may vary from line to line.

\begin{proof}[Proof of Theorem~\ref{thm23}]
The smoothing estimates in Proposition~\ref{prop3} guarantee that 
--- in order to estimate $\|u(t)\|_X$ --- 
it is sufficient to estimate $\|u(t)\|_r$ for some $r>r^*$,
where $r^*=\frac32(q-1)\max(1,\frac3q)<6$.
Consequently, we may assume $R<6$. In particular,
$H^{1,2}(\Omega)\hookrightarrow L_R(\Omega)=X$ and this embedding is compact.

{\it Step 1.}\/ Assume $q<3$. We will show that if
$\tau\in(0,\infty)$, $\tau\leq T_{max}(u_0)$, and
\begin{equation} \label{Etau}
E(u(t))\geq C_0\ \hbox{ for }\ t\in(0,\tau),
\end{equation}
then 
\begin{equation} \label{AEtau}
\|u(t)\|_X\leq C=C(\|u_0\|_X,C_0,\tau)\ \hbox{ for }\ t\in(0,\tau).
\end{equation}
Proposition~\ref{prop3} implies the existence of $\delta\in(0,\tau)$
and $C_1=C_1(\|u_0\|_X,\tau)$ 
such that 
\begin{equation} \label{estdelta}
 \|u(t)\|_X\leq C(\|u_0\|_X)\ \hbox{ for }\ t\in[0,2\delta]
\end{equation}
and
\begin{equation} \label{T1}
\|u(\delta)\|_\infty+\|u(\delta)\|_{1,2}+E(u(\delta))\leq C_1.
\end{equation}
Since 
$$ \frac{d}{dt}E(u(t))=-\int_\Omega u_t^2(x,t)\,dx, $$
we have 
\begin{equation} \label{estut}
\int_{\delta}^{\tau}\|u_t\|_2^2\,dt\le C_1-C_0,
\end{equation} 
hence $\|u(t)\|_2\leq C$ on $[\delta,\tau)$,
where $C=C(C_1,C_0,\tau)=C(\|u_0\|_X,C_0,\tau)$.
Set $\theta:=\frac{q+1}{6(q-1)}\in(0,1)$ and consider $t\in[\delta,\tau)$. Then $4\theta<q+1$, hence
the Hardy-Littlewood-Sobolev inequality, interpolation, the estimate of $\|u(t)\|_2$ above
and the H\"older inequality imply
\begin{equation} \label{IHLS}
 \begin{aligned}
I(t)&:=\int_\Omega\int_\Omega\frac{u^2(x,t)u^2(y,t)}{|x-y|}\,dy\,dx
 \leq C\|u(t)\|_{12/5}^4 \\
 &\leq C\|u(t)\|_2^{4(1-\theta)}\|u(t)\|_{q+1}^{4\theta}
 \leq C\|u(t)\|_{q+1}^{4\theta}
 \leq \eps\|u(t)\|_{q+1}^{q+1}+C_\eps.
\end{aligned}
\end{equation}
Fix $\eps<\frac{2(q-1)}{\lambda(q+1)}$.
Multiplying the equation in \eqref{pbmSPS} with $u$, integrating over $\Omega$ 
and using \eqref{IHLS} we obtain
\begin{equation}\label{multu}
\begin{aligned} 
\int_\Omega u(t)u_t(t)\,dx
  &= -\|u(t)\|_{1,2}^2+\|u(t)\|_{q+1}^{q+1}-\lambda I(t) \\
 &= -2E(u(t))+\frac{q-1}{q+1}\|u(t)\|_{q+1}^{q+1}-\frac\lambda 2 I(t) \\
 &\geq -2C_1+\Bigl(\frac{q-1}{q+1}- \frac{\lambda\eps}2\Bigr)\|u(t)\|_{q+1}^{q+1} 
 -\frac\lambda2 C_\eps \geq c\|u(t)\|_{q+1}^{q+1}-C.
\end{aligned}
\end{equation}
Consequently,
$$\int_{\delta}^{\tau}\|u(t)\|_{q+1}^{2(q+1)}\,dt
 \leq C\Bigl(1+\Bigl(\int_{\delta}^{\tau}\int_\Omega u^2(x,t)\,dx\,dt\Bigr)
   \Bigl(\int_{\delta}^{\tau}\int_\Omega u_t^2(x,t)\,dx\,dt\Bigr)\Bigr) \leq C.$$
Interpolation between this estimate and \eqref{estut} (see \cite{CL84,Q99}) yields
$\|u(t)\|_r\leq C$ for any $t\in[\delta,\tau)$ and $r<\frac{2q+4}3$.
Since $\frac{2q+4}3>\frac92\frac{q-1}q$, Proposition~\ref{prop3} and \eqref{estdelta} imply
\eqref{AEtau}.
Notice that this estimate also implies
\begin{equation} \label{AEinfty}
\lim_{t\to T_{max}(u_0)}E(u(t))>-\infty \quad\Rightarrow\quad T_{max}(u_0)=\infty.
\end{equation}

{\it Step 2.}\/
Assume $q\le3$.
If $t\in(0,T_{max}(u_0))$, then we will prove estimate \eqref{estEut} below.
Without loss of generality we may assume that $\Omega$ is starshaped with respect to zero.
Multiplying the equation in \eqref{pbmSPS} with $x\cdot\nabla u$ and integrating by parts 
(cf.~\cite{CSS10} and \cite[Theorem 5.1]{QS19})
we obtain the following Pohozaev type identity
$$ 
\begin{aligned}
\frac12\|u(t)\|_{1,2}^2+\|u(t)\|_2^2&-\frac3{q+1}\|u(t)\|_{q+1}^{q+1}+\frac54\lambda I(t)  \\
 &=-\int_{\partial\Omega}\Big|\frac{\partial u}{\partial\nu}\Big|^2x\cdot\nu\,dS+\int_\Omega u_t(t)x\cdot\nabla u(t)\,dx,
\end{aligned}
$$ 
where $\nu$ denotes the outer unit normal on $\partial\Omega$.
Since $x\cdot\nu\geq 0$ on $\partial\Omega$ and 
$$\int_\Omega u_t(t)x\cdot\nabla u(t)\,dx\leq C\|\nabla u(t)\|_2\|u_t(t)\|_2\leq\eps\|u(t)\|_{1,2}^2+C_\eps\|u_t(t)\|_2^2,$$
the Pohozaev identity guarantees
\begin{equation} \label{Poh1}
\frac12\|u(t)\|_{1,2}^2+\|u(t)\|_2^2-\frac3{q+1}\|u(t)\|_{q+1}^{q+1}+\frac54\lambda I(t)
\le \eps\|u(t)\|_{1,2}^2+C_\eps\|u_t(t)\|_2^2.
\end{equation}
Next the first line in \eqref{multu} similarly implies
\begin{equation} \label{multu2}
-\|u(t)\|_{1,2}^2+\|u(t)\|_{q+1}^{q+1}-\lambda I(t)=\int_\Omega u(t)u_t(t)\,dx
\leq \eps\|u(t)\|_{1,2}^2+C_\eps\|u_t(t)\|_2^2,
\end{equation}
and \eqref{E-SPS} yields 
\begin{equation} \label{Eest}
\frac12\|u(t)\|_{1,2}^2-\frac1{q+1}\|u(t)\|_{q+1}^{q+1}+\frac\lambda4 I(t)=E(u(t)). 
\end{equation}
Choose $\eps=\frac1{12}(q-2)$, hence $C_\eps=C(q)$.
Multiplying inequalities \eqref{multu2} and \eqref{Eest} by 2 and q+1 respectively, and adding 
them to \eqref{Poh1} we obtain (cf.~\cite[(19)]{GW20})
\begin{equation} \label{estEut}
 \frac{q-2}4\|u(t)\|_{1,2}^2+\frac{q-2}{q+1}\|u(t)\|_{q+1}^{q+1}+ \frac{q-2}4\lambda I(t)
 \le  3C(q)\|u_t(t)\|_2^2+(q+1)E(u(t)).
\end{equation}

{\it Step 3.}\/
Assume that $q<3$ and
\begin{equation} \label{C0}
C_0:=\lim_{t\to T_{max}(u_0)}E(u(t))>-\infty
\end{equation}
and notice that \eqref{AEinfty} implies $T_{max}(u_0)=\infty$.
We will show that
\begin{equation} \label{limsupuX}
\limsup_{t\to\infty}{\|u(t)\|_X}<C=C(C_0).
\end{equation}
Choose $T_0>0$ such that $E(u(T_0))<C_0+1$.
Then
\begin{equation} \label{estut2}
\int_{T_0}^{\infty}\|u_t\|_2^2\,dt\le 1.
\end{equation} 
If 
\begin{equation} \label{assut1}
\hbox{$t_1>T_0$ \ and \ $\|u_t(t_1)\|_2\leq 1$}, 
\end{equation}
then the inequality $E(u(t_1))\leq C_0+1$ and estimate \eqref{estEut}
imply
$\|u(t_1)\|_{1,2}\leq C(C_0)$.
hence 
$\|u(t_1)\|_X\leq C(C_0)$.
Proposition~\ref{prop3} implies the existence of $\delta^*>0$ and 
$C_2=C_2(\|u(t_1)\|_X)=C_2(C_0)>0$ 
such 
\begin{equation} \label{estt1delta}
\hbox{$\|u(t)\|_X\leq C_2$ for $t\in[t_1,t_1+\delta^*]$.}
\end{equation}
  
Estimate \eqref{estut2} implies the existence of $T_1>T_0$ with the following property:
\begin{equation} \label{utsmall}
\hbox{For any $\tilde t\ge T_1$ there exists $t_1\in(\tilde t,\tilde t+\delta^*/2)$
such that $\|u_t(t_1)\|_2\le1$.}
\end{equation}
Consequently, \eqref{estt1delta} implies 
\begin{equation} \label{estT1}
\hbox{$\|u(t)\|_X\leq C_2$ for all $t>T_1+\delta^*$,}
\end{equation}
hence \eqref{limsupuX} is true.

{\it Step 4.}\/
Assume $q<3$.
We will prove that \eqref{E-pos} implies \eqref{AEpos}, hence assertion (i)
is true.
Assume on the contrary that there exist solutions $u_k$ with initial data $u_{0,k}\in X$ such that
$\|u_{0,k}\|_X\leq C$ and
$E(u_k(t))\geq0$ for $t\in(0,T_{max}(u_{0,k}))$, but
$\|u_k(t_k)\|_X\to\infty$ for some $t_k\in(0,T_{max}(u_{0,k}))$.
Notice that $T_{max}(u_{0,k})=\infty$ due to \eqref{AEinfty} 
and $t_k\to\infty$ due to \eqref{AEtau}.
Estimate \eqref{AEtau} also shows that $M_k:=\inf_{t\in[t_k-1,t_k]}\|u_k(t)\|_X\to\infty$
as $k\to\infty$. Since $E(u_k(t))\leq C$ on $[t_k-1,t_k]$ due to \eqref{T1},
estimate \eqref{estEut} implies $\inf_{t\in[t_k-1,t_k]}\|\frac{\partial u_k}{\partial t}(t)\|_2\to\infty$,
hence $E(u_k(t_k-1))\geq E(u_k(t_k-1))-E(u_k(t_k))\to\infty$, which yields a contradiction.
 
{\it Step 5.}\/
Assume that \eqref{E-pos} fails and $q<3$.
Then $E(u(t^*))<0$ for some $t^*>0$.
If $C_0:=\lim_{t\to T_{max}(u_0)}E(u(t))>-\infty$, then \eqref{AEinfty} and \eqref{limsupuX} 
imply
$$T_{max}(u_0)=\infty\ \hbox{ and }\ \limsup_{t\to\infty}\|u(t)\|_X<\infty.$$
Proposition~\ref{prop3} guarantees $\limsup_{t\to\infty}\|u(t)\|_{1,2}<\infty$
and the compactness of the embedding
$H^1_2(\Omega)\hookrightarrow X$ implies that the $\omega$-limit set $\omega(u_0)$ of the solution $u$ in $X$ is nonempty and
consists of steady states (due to the existence of the Lyapunov functional $E$).
Choose $u_S\in\omega(u_0)$. Then $E(u_S)\le E(u(t^*))<0$ and \eqref{estEut} 
with $u(t)$ and $u_t(t)$ replaced by $u_S$ and $0$, respectively,
yields a contradiction.
Consequently,
$\lim_{t\to T_{max}(u_0)}E(u(t))=-\infty$.
Since $|E(u(t))|\leq C(\|u(t)\|_{1,2})\leq C(\|u(t-\delta)\|_X)$, we also have 
$\lim_{t\to T_{max}(u_0)}\|u(t)\|_X=\infty$ which concludes the proof of
(ii) if $q<3$.

Notice also that \eqref{estEut} implies $\frac{d}{dt}(-E(u(t)))=\|u_t(t)\|_2^2>-cE(u(t))$. 

{\it Step 6.}\/
Let $q\ge3$.
If $E$ stays bounded below, then the results in \cite{I12}
imply both $T_{max}(u_0)=\infty$ and estimate \eqref{AE2},
hence it is sufficient to prove assertion (ii).   
Assume that \eqref{E-pos} fails. 
Then \eqref{E-neg} follows again from \cite{I12} 
and the smoothing estimates in Proposition~\ref{prop3}.
Hence it is sufficient to show \eqref{TfiniteAE}.

Fix $t_0\in(0,T_{max}(u_0))$
such that $E(u(t_0))<0$ and let $\delta\in(0,T_{max}(u_0)-t_0)$,
$t_\delta:=T_{max}(u_0)-\delta$ if $T_{max}(u_0)<\infty$, $t_\delta:=t_0$
otherwise.
Set
$$ y(t):=\|u(t)\|_2^2,\quad
  w(t):=\|u(t)\|_{q+1}^{q+1},\quad
  z(t):=-E(u(t)) $$
and consider $t\geq t_0$.
Notice that $w\geq cy^{(q+1)/2}$ and $z(t)\geq z(t_0)>0$. 
In addition, $z_t(t)=\int_\Omega u_t^2(t)\,dx$.

First assume $q>3$. 
Multiplying the equation in \eqref{pbmSPS} with $u$ and integrating over $\Omega$ we obtain
(cf.~\eqref{multu})
\begin{equation} \label{yzy}
\begin{aligned} 
\frac12y_t(t) &=\int_\Omega u(t)u_t(t)\,dx = -\|u(t)\|_{1,2}^2+w(t)-\lambda I(t) \\
&= 4z(t)+\|u(t)\|_{1,2}^2+\Bigl(1-\frac4{q+1}\Bigr)w(t) 
  \geq 4z(t_\delta)+ cy^{(q+1)/2}(t), \quad t\geq t_\delta, 
\end{aligned}
\end{equation}
which implies $T_{max}(u_0)<\infty$ and
$-E(u(t))\le -E(u(t_\delta))=z(t_\delta)<C(\delta)$ for $t\le t_\delta$.
This estimate and estimates in \cite{I12} imply \eqref{AE}.

Next assume $q=3$. If $\lambda$ is small enough, then one can use a slight modification
of the simple argument used in the case $q>3$ (see \cite{LM15}).
However, without such restriction on $\lambda$ we have to use more complicated arguments.
First notice that \eqref{Poh1},\eqref{multu2} and \eqref{Eest} imply
\begin{equation} \label{Poh1add}
\frac54\lambda I(t)-\frac34w(t)\leq Cz_t(t),
\end{equation}
\begin{equation} \label{multu2add}
-\|u(t)\|_{1,2}^2+w(t)-\lambda I(t)=\int_\Omega u(t)u_t(t)\,dx,
\end{equation}
and
\begin{equation} \label{Eestadd}
\frac12\|u(t)\|_{1,2}^2-\frac14w(t)+\frac14\lambda I(t)=-z(t).
\end{equation}
Formula \eqref{Eestadd} implies $w>\lambda I$, hence there exists $\eps(t)\in(0,1)$ such that
\begin{equation} \label{zw}
 \eps(t)w(t)=w(t)-\lambda I(t)=2\|u(t)\|_{1,2}^2+4z(t)>2\|u(t)\|_{1,2}^2.
\end{equation}
Now \eqref{zw} and \eqref{multu2add} yield
$$ \begin{aligned}
\frac{\eps(t)}2w(t) &\leq -\|u(t)\|_{1,2}^2+\eps(t)w(t)
=-\|u(t)\|_{1,2}^2+w(t)-\lambda I(t)
\\
&=\int_\Omega u(t)u_t(t)\,dx
  \leq \Bigl(\int_\Omega u^4(t)\,dx\Bigr)^{1/4}\Bigl(\int_\Omega |u_t|^{4/3}\,dx\Bigr)^{3/4}
 \leq Cw(t)^{1/4}z_t(t)^{1/2}. 
\end{aligned}$$
This estimate and \eqref{zw} guarantee
\begin{equation} \label{estzt1}
   z_t(t)\geq c\eps^2(t)w^{3/2}(t)\geq c\eps^{1/2}(t)z^{3/2}(t),
\end{equation}
hence 
\begin{equation} \label{estzt1a}
   z_t(t)\geq cz^{3/2}(t)\quad\hbox{provided }\ \eps(t)>\frac14.
\end{equation}
If $ \eps(t)\le\frac14$,
then $\lambda I(t)\geq\frac34 w(t)$, hence
\eqref{Poh1add} and \eqref{zw}
imply
\begin{equation} \label{estzt2}
 z_t(t)\geq cw(t)\geq c\frac1{\eps(t)}z(t) \quad \hbox{provided }\ \eps(t)\le\frac14.
\end{equation}
Now \eqref{estzt1} and \eqref{estzt2} yield
\begin{equation} \label{estzt2a}
 z_t(t)\geq c(\eps^{1/2}(t)z^{3/2}(t))^{2/3}\Bigl(\frac1{\eps(t)}z(t)\Bigr)^{1/3}=c z^{4/3}(t),\quad 
\hbox{provided }\ \eps(t)\le\frac14.
\end{equation}
Similarly as in the case $q>3$,
inequalities \eqref{estzt1a} and \eqref{estzt2a} imply 
$T_{max}(u_0)<\infty$ and $z(t_\delta)<C_\delta$,
and this estimate together with estimates in \cite{I12} imply \eqref{AE}.
\end{proof}

\begin{proof}[Proof of Corollary~\ref{cor23}]
Without loss of generality we may assume $0\in\Omega$.
Fix a positive integer $k$, $\delta\in(0,(q-2)/(2k))$, consider $M\gg1$ and set
$$M_i:=M^{1+(i-1)\delta},\quad  v_i(x):=M_i^2(1-M_i|x|)_+,\quad i=1,2,\dots,k.$$
By $c,C$ we denote positive constants which may vary from line to line,
but which are independent of $M$.
Using the first inequality in \eqref{IHLS} we obtain
$$\|v_i\|_{q+1}^{q+1}\ge cM_i^{2q-1}\quad\hbox{and}\quad
 \int_\Omega(|\nabla v_i|^2+v_i^2)\,dx
 +\lambda\int_\Omega\int_\Omega\frac{v_i^2(x)v_i^2(y)}{|x-y|}\,dy\,dx\le CM_i^3.$$
Let 
\begin{equation} \label{v}
v:=\sum_{i=1}^k\alpha_iv_i, \quad \hbox{where }\  \sum_{i=1}^k|\alpha_i|=1,
\end{equation}
and set $N_i:=\|\alpha_iv_i\|_{q+1}^{q+1}$.
Then $N_{i_1}\ge N_{i_2}\ge \dots\ge N_{i_k}$ for some permutation $\{i_1,i_2,\dots,i_k\}$
of indices $\{1,2,\dots,k\}$.
Since $M_i\ge M$ for all $i$ and there exists $j$ such that $|\alpha_j|\geq1/k$, we have $N_{i_1}\geq cM^{2q-1}$.
Set $A:=\{j:N_{i_j}\ge M^\delta N_{i_{j+1}}\}$ and
$m:=\min A$ if $A$ is nonempty, $m:=k$ otherwise.
Set also $i^*:=\min\{i_1,i_2,\dots,i_m\}$.
If $j\leq m$ and $i_j\ne i^*$, then $i_j>i^*$, hence $M_{i_j}\geq M_{i^*}M^\delta$
and $v_{i_j}(x)=0$ if $|x|>M_{i^*}^{-1}M^{-\delta}$.
Consequently,
$$ \Big\|\sum_{j=1}^m\alpha_{i_j}v_{i_j}\Big\|_{q+1}^{q+1}
  \geq\int_{M_{i^*}^{-1}M^{-\delta}<|x|<M_{i^*}^{-1}}|\alpha_{i^*}u_{i^*}|^{q+1}\,dx
 \ge cN_{i^*},$$
while the definition of $m$ implies
$$ \Big\|\sum_{j=m+1}^k\alpha_{i_j}v_{i_j}\Big\|_{q+1}^{q+1}\le CN_{i^*}M^{-\delta}, $$
hence the inequality $$\|w+z\|_{q+1}^{q+1}\ge c\|w\|_{q+1}^{q+1}-C\|z\|_{q+1}^{q+1}$$
guarantees
$$ \|v\|_{q+1}^{q+1}=\Big\|\sum_{j=1}^k\alpha_{i_j}v_{i_j}\Big\|_{q+1}^{q+1}
  \ge cN_{i^*} \ge cN_{i_1}M^{-(m-1)\delta}\ge cM^{2q-1-k\delta}.$$
Since
$$ \int_\Omega(|\nabla v|^2+v^2)\,dx+\lambda\int_\Omega\int_\Omega\frac{v^2(x)v^2(y)}{|x-y|}\,dy\,dx\le CM_k^3\le CM^{3(1+k\delta)}, $$
and $2q-1-k\delta>3(1+k\delta)$, we have  
$E(v)<0$ provided $M$ is large enough.
Fix such $M$.

Set $X_k:=\hbox{span}\{v_1,v_2,\dots,v_k\}$,
$$O_k:=\{u_0\in X_k: \hbox{the solution $u$ of \eqref{pbmSPS} exists globally and tends to 0 as $t\to\infty$}\},$$
and let $D_k$ be the component of $O_k$ containing $u_0=0$.
Since $u=0$ is an asymptotically stable steady state of \eqref{pbmSPS}, the sets $O_k,D_k$ are open. 
Since $E(0)=0$ and $E(v)<0$ for any $v$ satisfying \eqref{v}, we have
$v\notin O_k$, hence $D_k$ is bounded. 
Now the same arguments as in \cite{Q04} show that we may find $u_0,\tilde u_0\in\partial D_2$ such that
the corresponding solutions $u$ and $\tilde u$ are global, bounded and their $\omega$-limit sets
contain a positive and a sign-changing steady state,\ respectively.
Similarly, if $\Omega$ is a ball with radius $R$ and center at zero, then each $v\in X_k$ is a radial function
which changes sign at most $k-1$ times (since $|x|\mapsto v(|x|)$ is affine on intervals
of the form $[0,1/M_k]$ or $[1/M_{i+1},1/M_i]$, and $v=0$ on $[1/M_1,R]$). Consequently, the arguments in \cite{Q04} (cf.~also \cite{I13})
guarantee that for any $i=0,\dots,k-1$ we can find $u_0\in\partial D_k$ such that
the $\omega$-limit set of the corresponding solution contains a steady state with precisely $i$ sign changes.
\end{proof}

\subsection{Proofs of Theorems~\ref{thmChoquard} and \ref{thmFracLap}}

Assume \eqref{assOmega}. Let
$\alpha\in(0,1]$, ${\mathcal A}=(-\Delta)^\alpha$, 
$1<p<p_S(\alpha)$ and $u_0\in L_q(\Omega)$, where $q>\max(1,\frac{n}{2\alpha}(p-1))$.

If $\alpha<1$, then set $c_p:=\frac1{p+1}$, ${\mathcal F}(u):=|u|^{p-1}u$  
and notice that ${\mathcal F}=\nabla\Phi$, where
$$ \Phi(u) = c_p\int_\Omega|u|^{p+1}\,dx. $$
Next assume $\alpha=1$, $p\geq2$, $3\le n\le5$, set $c_p:=\frac1{2p}$, 
\begin{equation} \label{vu}
{\mathcal F}(u):=v_u|u|^{p-2}u,\quad\hbox{where}\quad
v_u(x):=\int_\Omega\frac{|u(y)|^p}{|x-y|^{n-2}}\,dy,
\end{equation}
and notice that ${\mathcal F}=\nabla\Phi$, where
$$ \Phi(u) = c_p\int_\Omega\int_\Omega\frac{|u(x)|^p|u(y)|^p}{|x-y|^{n-2}}\,dy\,dx. $$ 
Denoting by $\langle\cdot,\cdot\rangle_\Omega$ the duality in $L_2(\Omega)$
and by $\|\cdot\|_{2,\R^n}$ the norm in $L^2(\R^n)$, we obtain
$$\begin{aligned}
\|u\|_{p+1}^{p+1}&=\langle u,|u|^{p-1}u\rangle_\Omega=\langle u,|u|^{p-1}u\rangle_{\R^n} \\
&=\langle (-\Delta)^{1/2}u,(-\Delta)^{-1/2}|u|^{p-1}u\rangle\rangle_{\R^n}
\leq\|(-\Delta)^{1/2}u\|_{2,\R^n}\|(-\Delta)^{-1/2}|u|^{p-1}u\|_{2,\R^n} \\
&\leq \|\nabla u\|_{2,\R^n}^2+\|(-\Delta)^{-1/2}|u|^{p-1}u\|_{2,\R^n}^2
\leq C(\int_\Omega{\mathcal A}u\cdot u\,dx+\Phi(u)),
\end{aligned}$$  
where we used (see \cite[5.10(2)]{LL01}) 
$$ \|(-\Delta)^{-1/2}f\|_{2,\R^n}^2=c(n)\int_{\R^n}\int_{\R^n}\frac{f(x)f(y)}{|x-y|^{n-2}}\,dy\,dx. $$
Consequently, both for $\alpha<1$ and $\alpha=1$ we have
\begin{equation} \label{fp1}
\int_\Omega{\mathcal A}u\cdot u\,dx+\Phi(u) \geq c\|u\|_{p+1}^{p+1}\geq \tilde c\|u\|_2^{p+1}.
\end{equation}
Set also
$$E(u):=\frac12\int_\Omega{\mathcal A}u\cdot u\,dx 
-\Phi(u)$$
and notice that $\Phi(u)=c_p\int_\Omega{\mathcal F}(u)u\,dx$.
Choose $c_1\in(2,p+1)$.
Multiplying the equation in \eqref{pbmChoquard} or \eqref{fracLap} with $u$,
integrating over $\Omega$ and using \eqref{fp1} we obtain
\begin{equation} \label{uut}
\begin{aligned}
\frac12\frac{d}{dt}&\|u(t)\|_2^2 
=\int_\Omega u_tu\,dx=\int_\Omega(-{\mathcal A}(u)+{\mathcal F}(u))u\,dx \\
&\geq-c_1E(u)+c_2(\|u\|_{\alpha,2}^2+\Phi(u))
\geq-c_1E(u(t))+c_3\|u(t)\|_2^{p+1}, 
\end{aligned}
\end{equation}
where $c_2,c_3>0$.
Fix $\delta>0$ small.
Since $E(u(\delta))\leq \|u(\delta)\|_{\alpha,2}^2\leq C(\|u_0\|_q)$,
estimate \eqref{uut} implies 
\begin{equation} \label{boundE}
C\geq E(u(\delta)\geq E(u(t))\geq-C\quad\hbox{and}\quad
 \|u(t)\|_2\leq C\quad\hbox{for }\ t\in(\delta,T_\delta),
\end{equation}
where $T_\delta:=T-\delta$ if $T<\infty$, $T_\delta=\infty$ if $T=\infty$,
and $C=C(\delta,c_1,c_3,p)$
(otherwise the function $y(t):=\|u(t)\|_2^2$ would blow up at some $T_y<T$). 
In addition,
\begin{equation} \label{ut2}
\int_\delta^{T_\delta}\int_\Omega u_t^2(x,t)\,dx\,dt
\leq E(u(\delta))-E(u(T_\delta))\leq C,
\end{equation}
and the boundedness of energy and \eqref{fp1} also imply
\begin{equation} \label{phip1}
C+\Phi(u) \geq c\|u\|_{p+1}^{p+1}.
\end{equation}
If
$\delta\le t_1<t_2<T_\delta$, \ $t_2-t_1\le 2$ and $Q\geq2$, then \eqref{phip1}, \eqref{boundE} and \eqref{uut} imply   
\begin{equation} \label{estLong}
 \int_{t_1}^{t_2}\|u(t)\|_{p+1}^{Q(p+1)}\,dt
 \leq C\Bigl(1+\int_{t_1}^{t_2}\Phi(u(t))^Q\,dt\Bigr)
\leq C\int_{t_1}^{t_2}\Bigl(1+\Big|\int_\Omega u_tu\,dx\Big|^Q\Bigr)\,dt,
\end{equation}
and if $Q=2$, then using the Cauchy inequality,
 the boundedness of $\|u(t)\|_2$ and \eqref{ut2} we also obtain
$$ \int_{t_1}^{t_2}\Bigl(1+\Big|\int_\Omega u_tu\,dx\Big|^2\Bigr)\,dt \leq C, $$
hence 
\begin{equation} \label{boundQ}
 \int_{t_1}^{t_2}\Phi(u(t))^Q\,dt \leq C
\end{equation}
provided $Q=2$.
The estimates above show that if $Q=2$, 
then $u$ is bounded in $W^1_2((t_1,t_2),L_2(\Omega))\cap L_{Q(p+1)}((t_1,t_2),L_{p+1}(\Omega))$ 
and interpolation (see \cite[Proposition~2.1]{Q03} and \cite[Theorems~VII.2.7.2 and VII.2.8.3]{A19}) yields
\begin{equation} \label{boundLs}
 \|u(t)\|_s\leq C \quad\hbox{if}\quad s< s_Q:=p+1-\frac{p-1}{Q+1}.
\end{equation}
If $s_Q>s^*=s^*(\alpha):=\frac{n}{2\alpha}(p-1)$, then estimate
\eqref{boundLs} 
and Propositions~\ref{prop2} and~\ref{prop4}
imply the claimed estimate \eqref{AE} with $X:=L_s$ and $s\in(s^*,s_Q)$.
Since estimate \eqref{boundLs} is guaranteed by \eqref{boundQ} and
$\lim_{Q\to\infty}s_Q=p+1>s^*$, it is sufficient to prove \eqref{boundQ} with $Q$ large enough
($Q>Q_p$, where $Q_p$ is such that $s_{Q_p}=s^*$, i.e.\ $\frac1{Q_p+1}=\frac{p+1}{p-1}-\frac{n}{2\alpha}$).

We will use a bootstrap argument. We will assume that \eqref{boundQ} (hence
also \eqref{boundLs}) is true for some $Q\in[2,Q_p]$ and we will prove that there
exists $\tilde Q>Q$ such that \eqref{boundQ} is true with $Q$ replaced by $\tilde Q$.
In addition, the difference $\tilde Q-Q$ will be bounded below by a positive constant. 

Hence assume \eqref{boundQ} and \eqref{boundLs} for some $Q\in[2,Q_p]$ and consider $\tilde Q>Q$.
Consider $s\in(2,s_Q)$ close to $s_Q$, set $s'=\frac s{s-1}$ and 
$$r=\begin{cases} r_1:=\frac{p+1}p &\hbox{ if }\alpha<1,\\
                  r_2:=\frac{2np}{2np-n-2} &\hbox{ if }\alpha=1.
\end{cases}$$
Recall that we assume $n\ge3$ and $2\le p<p^*<p_S$ if $\alpha=1$,
hence $r_2$ is well defined and $\frac{p}{p-1}>r_2>r_1>1$. 
In addition,  if $s<s_Q\le s^*(\alpha)$, then $p<p_S(\alpha)$ implies $r_i<s'$ for $i=1,2$.
Let $\theta\in(0,1)$ be defined by $\frac\theta{r}+\frac{1-\theta}2=\frac1{s'}$
and $\beta:=\frac2{(1-\theta)\strut\tilde Q}$.
Notice that if $\tilde Q$ is sufficiently close to $Q$ and $s$ is sufficiently close to $s_Q$,
then $\beta>1$.

Let $D_r$ be the domain of the Dirichlet fractional Laplacian $(-\Delta)^\alpha_D$
considered as an unbounded operator in $L_r(\Omega)$.
Given an interval of the form $[t_2-\tau,t_2]\subset[\delta,T_\delta)$ with $\tau\in[2\delta,2]$,
Propositions~\ref{prop2} and~\ref{prop4} guarantee that we can
choose $t_1\in[t_2-\tau,t_2-\tau+\delta]$ such that $\|u(t_1)\|_Y\le C$, where $Y$ is the real interpolation space
$(L_r(\Omega),D_r)_{1-1/\tilde p,\tilde p}$
with $\tilde p:=\beta'\theta\tilde Q$
(cf.\ the space $X_{p,q}$ in the proof of Proposition~\ref{prop4}).

Now using gradually the second inequality in \eqref{estLong},
H\"older's inequality and \eqref{boundLs}, interpolation,
H\"olders's inequality and \eqref{ut2}, and the maximal regularity estimate
\eqref{MRfrac}
we obtain
$$ \begin{aligned} 
 \int_{t_1}^{t_2}\Phi(u(t))^{\tilde Q}\,dt
&\leq C\Bigl(1+\int_{t_1}^{t_2}\Big|\int_\Omega u_tu\,dt\Big|^{\tilde Q}\Bigr) \\
&\leq C\Bigl(1+\int_{t_1}^{t_2}\|u_t\|_{s'}^{\tilde Q}\,dt\Bigr) \\
&\leq C\Bigl(1+\int_{t_1}^{t_2}\|u_t\|_{r}^{\theta\tilde Q}\|u_t\|_2^{(1-\theta)\tilde Q}\,dt\Bigr) \\
&\leq C\Bigl(1+\Bigl(\int_{t_1}^{t_2}\|u_t\|_{r}^{\beta'\theta\tilde Q}\,dt\Bigr)^{1/\beta'}\Bigr) \\
&\leq C\Bigl(1+\Bigl(\int_{t_1}^{t_2}\|{\mathcal F}(u)\|_{r}^{\beta'\theta\tilde Q}\,dt\,\Bigr)^{1/\beta'}\Bigr). \\
\end{aligned} $$
Hence estimate \eqref{boundQ} with $Q$ replaced by $\tilde Q$ is true provided
we can verify the  bootstrap condition
\begin{equation} \label{BC}
\|{\mathcal F}(u)\|_{r}^{\beta'\theta\tilde Q}\leq C(1+\Phi(u)^{\tilde Q}).
\end{equation}
If $\alpha<1$, then  $r=\frac{p+1}p$ and ${\mathcal F}(u)=|u|^{p-1}u$,
hence $\|{\mathcal F}(u)\|_{r}=C\Phi(u)^{\frac p{p+1}}$
and the condition \eqref{BC} is true for any 
$\tilde Q<Q+2$ (cf.\ \cite[the proof of Theorem 22.1]{QS19}).

Next assume $\alpha=1$, hence $r=\frac{2np}{2np-n-2}$ and
${\mathcal F}(u)=\bigl(\int_\Omega\frac{|u(y)|^p}{|x-y|^{n-2}}\,dy\bigr)|u|^{p-2}u$.
First we will prove that
\begin{equation} \label{FPhi}
\|{\mathcal F}(u)\|_{r}\le \Phi(u)^{\frac{2p-1}{2p}}.
\end{equation}
In fact, using the notation in \eqref{vu} we obtain
$$ \begin{aligned}
\|{\mathcal F}(u)\|_{r}&=\Bigl(\int_\Omega (v_u|u|^{p-1})^r\,dx\Bigr)^{1/r}
=\Bigl(\int_\Omega v_u^{r(p-1)/p}|u|^{(p-1)r}v_u^{r/p}\,dx\Bigr)^{1/r} \\
&\leq \Bigl(\int_\Omega v_u|u|^p\,dx\Bigr)^{(p-1)/p}
    \Bigl(\int_\Omega v_u^{r/(p-r(p-1))}\,dx\Bigr)^{1/r-(p-1)/p} \\
 &= c\Phi(u)^{(p-1)/p}\|v_u\|_{2^*}^{1/p},
\end{aligned}
$$
where $2^*=\frac{r}{p-r(p-1)}=\frac{2n}{n-2}$.  
Since
$$ \|v_u\|_{2^*}^2\leq\int_{\R^n}|\nabla v_u|^2\,dx
  =\int_{\R^n} v_u(-\Delta v_u)\,dx
  =\int_{\R^n} v_u|u|^p\,dx = \Phi(u),$$
we obtain \eqref{FPhi}. 
Now \eqref{FPhi} guarantees that the bootstrap condition \eqref{BC}
is satisfied if $\frac{2p-1}{2p}\beta'\theta\leq 1$.
Hence it is sufficient to show 
\begin{equation} \label{bc}
(1-\xi)\beta'\theta<1,
\end{equation}
where $\xi:=\frac1{2p}$.
Since this is an open condition, it is sufficient to show
that \eqref{bc} is true if $\beta=\frac2{(1-\theta)Q}$ 
and $\frac\theta r+\frac{1-\theta}2=\frac1{s'}$,
where $r=r_2$ and $s=s_Q$.  In this case, \eqref{bc} is equivalent to
\begin{equation} \label{bcxi}
 2\xi>(Q-2)\Bigl(\Bigl(\frac2r-1\Bigr)\frac{s}{s-2}-1\Bigr),
\end{equation}
where $\xi=\frac1{2p}$, $r=r_2$ and $s=s_Q$.
This condition is equivalent to
$$ (p-1)B>-4(Q+1)(Q-2),\quad\hbox{where}\quad B:=nQ-(Q-2)(2n+Q(n-2)).$$
If $Q=2$ or $B\geq0$, then this condition is satisfied,
hence we may assume $Q>2$ and $B<0$. Then we obtain the following condition
$$p<\frac{Q^2(n+2)-nQ-4(n+2)}{Q^2(n-2)-(n-4)Q-4n}.$$
The minimum of the RHS in the last inequality is attained at $Q=n+2+\sqrt{n^2+3n}$
and equals $p^*$. Consequently, if $p<p^*$, then \eqref{AE} is true. 
The remaining assertions follow in the same way as in \cite{Q03} and \cite{Q04}.
\qed

\begin{remark} \rm
Consider $r=r_3$, where $\frac1{r_3}=\frac{p-1}s+\frac1{2^*}$ and $s<s_Q$.
Then one can use H\"older's inequality to estimate
$$
\|{\mathcal F}(u)\|_r=\Bigl(\int_\Omega (v_u|u|^{p-1})^r\,dx\Bigr)^{1/r}
\leq \|u\|_s^{p-1}\|v_u\|_{2^*} \leq C\Phi(u)^{1-\xi},
$$
where $\xi=\frac12$. If we used this estimate instead of \eqref{FPhi},
then the bootstrap condition \eqref{bc} (hence \eqref{bcxi} with
$\xi=\frac12$, $r=r_3$ and $s=s_Q$)
would yield the same restriction $p<p^*$.
\end{remark}

\vskip2mm
\noindent{\bf Acknowledgements}

\vskip1mm
\noindent The author was supported in part by the Slovak Research and Development Agency
under the contract No. APVV-23-0039 and by VEGA grant 1/0245/24.
He thanks Helmut Abels for helpful comments on \cite{AG24}.

\vskip2mm
\noindent{\bf Data availability} \
Data sharing is not applicable to this article as no datasets were
generated or analyzed during the current study.

\vskip2mm
\noindent{\bf Declarations}

\vskip1mm
\noindent{\bf Conﬂict of interest} \ The author has no conflict of interest to declare. 

\end{document}